\documentclass[10pt]{article}

\usepackage{tikz, float}
\usetikzlibrary{patterns}
\usepackage{amssymb}
\usepackage{url}
\usepackage{amscd}
\usepackage{amsthm}
\usepackage{amsmath}
\usepackage{graphicx}

\newtheorem {Theorem}  {Theorem}

\newtheorem {Problem} {Problem}

\newtheorem {Fact}{Fact}

\begin{document}
\baselineskip = 15pt
\bibliographystyle{plain}

\title{Undecidability of Translational Tiling with Three Tiles}
\date{}
\author{Chao Yang\\ 
              School of Mathematics and Statistics\\
              Guangdong University of Foreign Studies, Guangzhou, 510006, China\\
              sokoban2007@163.com, yangchao@gdufs.edu.cn\\
              \\
        Zhujun Zhang\\
              Big Data Center of Fengxian District, Shanghai, 201499, China\\
              zhangzhujun1988@163.com
              }
\maketitle

\begin{abstract}
Is there a fixed dimension $n$ such that translational tiling of $\mathbb{Z}^n$ with a monotile is undecidable? Several recent results support a positive answer to this question. Greenfeld and Tao disprove the periodic tiling conjecture by showing that an aperiodic monotile exists in sufficiently high dimension $n$ [Ann. Math. 200(2024), 301-363]. In another paper [to appear in J. Eur. Math. Soc.], they also show that if the dimension $n$ is part of the input, then the translational tiling for subsets of $\mathbb{Z}^n$ with one tile is undecidable. These two results are very strong pieces of evidence for the conjecture that translational tiling of $\mathbb{Z}^n$ with a monotile is undecidable, for some fixed $n$. This paper gives another supportive result for this conjecture by showing that translational tiling of the $4$-dimensional space with a set of three connected tiles is undecidable. 
\end{abstract}


\section{Introduction} 

We study the decidability or undecidability of the decision problem associated with the following translational tiling problem.

\begin{Problem}[Translational tiling of $\mathbb{Z}^n$ with a set of $k$ tiles] \label{pro_main}
A tile is a finite subset of $\mathbb{Z}^n$. Let $k$ and $n$ be fixed positive integers. Given a set $S$ of $k$ tiles in $\mathbb{Z}^n$, is there an algorithm to decide whether $\mathbb{Z}^n$ can be tiled by translated copies of tiles in $S$?
\end{Problem}

Note that Problem \ref{pro_main} can be stated equivalently and more geometrically as a tiling problem in $\mathbb{R}^n$, where a tile is the union of a finite set of unit hypercubes of the form $\Pi_{i=1}^n [z_i,z_i+1] $ ($z_i\in \mathbb{Z}$). A tile may not be connected in the Euclidean space $\mathbb{R}^n$. The answer to Problem \ref{pro_main} is known for some pairs of parameters $n$ and $k$. The problem is decidable for $n=1$ \cite{s93}, and for $(n,k)=(2,1)$ \cite{bn91,b20, gt21, w15}. On the other hand, the problem is undecidable for $n=2$ and $k\geq 8$ \cite{yang23, yang23b, yz24}, for $n=3$ and $k\geq 5$, and for $n=4$ and $k\geq 4$ \cite{yz24b,yz24c}. Two recent results of Greenfeld and Tao \cite{gt24a, gt24b} suggest that the translational tiling problem may be undecidable even for one tile ($k=1$) and some sufficiently large fixed dimension $n$. In fact, they disprove the periodic tiling conjecture \cite{gs16, lw96, s74} by showing the existence of an aperiodic monotile in some extremely large fixed dimension \cite{gt24a}. They also show that if the dimension $n$ is part of the input, the translational tiling for subsets of $\mathbb{Z}^n$ with one tile is undecidable \cite{gt24b}. 

The main contribution of this paper is to expand the knowledge of the undecidability of the translational tiling problem. In particular, we prove that translational tiling of the 4-dimensional space with three connected tiles is undecidable (Theorem \ref{thm_main}).

\begin{Theorem}[Undecidability with Three Tiles]\label{thm_main}
    Translational tiling of $4$-dimensional space with a set of $3$ polyhypercubes is undecidable.
\end{Theorem}

With Theorem \ref{thm_main}, the current state of knowledge on the decidability and undecidability of the translational tiling problem of $\mathbb{Z}^n$ with a set of $k$ tiles for fixed pairs $(k,n)$ can be summarized in Figure \ref{fig_nk}. The green region is known to be decidable, the red region is known to be undecidable, and the yellow region is possibly undecidable. Note that the frontier of the undecidable region is not yet clearly known, especially for the cases of one tile or two tiles. The boundary of the yellow region in Figure \ref{fig_nk} is to demonstrate the idea that as the dimension~$n$ increases, it may need fewer tiles to get undecidable results for translational tiling of $\mathbb{Z}^n$.

\begin{figure}[H]
\begin{center}
\begin{tikzpicture}[ pattern1/.style={draw=red,pattern color=yellow!70, pattern=north east lines}] ]

\draw [fill=green!20,dashed] (11.9,0)--(0,0)--(0,2)--(1,2)--(1,1)--(11.9,1);

\begin{scope}
    \clip (0,9) rectangle (2,10.9);
\filldraw [yellow!50] (0,11)--(0,8)--(3,8)--(3,11); 
\end{scope}

\begin{scope}
    \clip (2,1) rectangle (11.9,10.9);
\draw [fill=red!20,dashed] (12,1)--(7,1)--(7,2)--(3,2)--(3,3)--(2,3)--(2,4)--(2,12)--(12,12); 
\end{scope}

\begin{scope}
    \clip (0.64,3.9) rectangle (2,10.9);

\filldraw [yellow!50,dashed] plot [smooth] coordinates {(-1,11) (1,8.5) (1.5,6.5) (2,4)  (2.5,1) (9,0) (12,1) (12,15) (13,12)};
\end{scope}

\draw[help lines, color=gray, dashed]  (-0.1,-0.1) grid (11.9,10.9);
\draw[->,ultra thick] (-0.5,0)--(12,0) node[right]{$k$};
\draw[->,ultra thick] (0,-0.5)--(0,11) node[above]{$n$};

\foreach \y in {1,2,3,4,5}
{
\node at (-0.5,\y-0.5) {\y};
}
\node at (-0.9,9.7) {some};
\node at (-0.8,9.3) {large $n$?};

\foreach \x in {1,...,11}
{
\node at (\x-0.5,-0.5) {\x};
}

\end{tikzpicture}
\end{center}
\caption{Translational tiling problem of $\mathbb{Z}^n$ with a set of $k$ tiles.}\label{fig_nk}
\end{figure}
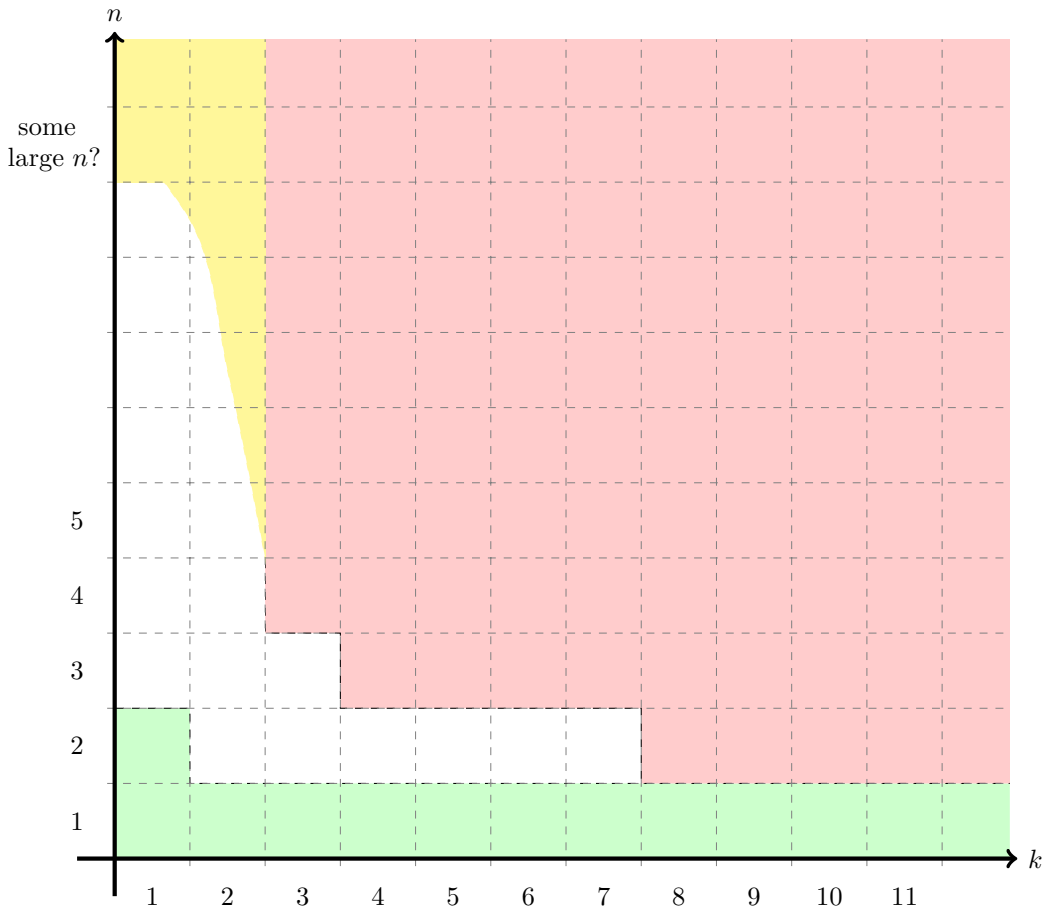

Before proving Theorem \ref{thm_main}, we first prove Theorem \ref{thm_3d_new} on the undecidability of translational tiling of $3$-dimensional space with $4$ tiles. Several novel techniques are introduced in the proof of Theorem \ref{thm_3d_new}. The main result, Theorem \ref{thm_main}, is then proved by lifting the $3$-dimensional tiling of Theorem \ref{thm_3d_new} to the $4$-dimensional space.

\begin{Theorem}[Undecidability with Four Tiles]\label{thm_3d_new}
    Translational tiling of $3$-dimensional space with a set of $4$ polycubes is undecidable.
\end{Theorem}

Like many other undecidable results on tiling problems \cite{dl24,gt23, gt24b,jr12,o09,yang23,yang23b,yz24,yz24a,yz24b,yz24c}, our proof relies on the undecidability of Wang's domino problem. A \textit{Wang tile} is a unit square with each edge assigned a color. Given a finite set of Wang tiles (see Figure \ref{fig_w3} for an example), Wang considered the problem of tiling the entire plane with translated copies of the set, under the conditions that the tiles must be edge-to-edge and the color of common edges of any two adjacent Wang tiles must be the same \cite{wang61}. This is known as \textit{Wang's domino problem}. Berger showed that Wang's domino problem is undecidable in general (i.e. the size of the set of Wang tiles can be arbitrarily large\footnote{If the size of the set of Wang tiles is fixed, then Wang's domino problem is decidable, as there are only a finite number of instances.}) in the 1960s.

\begin{Theorem}[\cite{b66}]\label{thm_berger}
    Wang's domino problem is undecidable.
\end{Theorem}


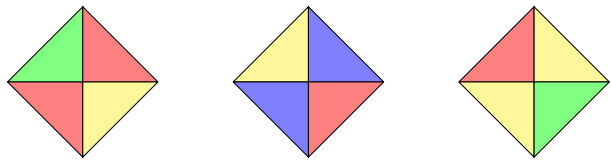
\begin{figure}[ht]
\begin{center}
\begin{tikzpicture}[scale=0.5]

\draw [ fill=green!50] (0,0)--(2,2)--(2,0)--(0,0);
\draw [ fill=red!50] (0,0)--(2,-2)--(2,0)--(0,0);
\draw [ fill=red!50] (4,0)--(2,2)--(2,0)--(4,0);
\draw [ fill=yellow!50] (4,0)--(2,-2)--(2,0)--(4,0);

\draw [ fill=yellow!50] (6+0,0)--(6+2,2)--(6+2,0)--(6+0,0);
\draw [ fill=blue!50] (6+0,0)--(6+2,-2)--(6+2,0)--(6+0,0);
\draw [ fill=blue!50] (6+4,0)--(6+2,2)--(6+2,0)--(6+4,0);
\draw [ fill=red!50] (6+4,0)--(6+2,-2)--(6+2,0)--(6+4,0);

\draw [ fill=red!50] (12+0,0)--(12+2,2)--(12+2,0)--(12+0,0);
\draw [ fill=yellow!50] (12+0,0)--(12+2,-2)--(12+2,0)--(12+0,0);
\draw [ fill=yellow!50] (12+4,0)--(12+2,2)--(12+2,0)--(12+4,0);
\draw [ fill=green!50] (12+4,0)--(12+2,-2)--(12+2,0)--(12+4,0);

\end{tikzpicture}
\end{center}
\caption{A set of $3$ Wang tiles} \label{fig_w3}
\end{figure}

The remainder of this paper is organized as follows. Section \ref{sec_3d} proves Theorem \ref{thm_3d_new} by introducing a new reduction framework that incorporates several new techniques. Section \ref{sec_4d} proves Theorem \ref{thm_main} by lifting the $3$-dimensional tile set and its tiling in the proof of Theorem \ref{thm_3d_new} to $4$-dimensional space. Section \ref{sec_conclu} concludes with a few remarks on future work.

\section{Undecidability of Tiling $3$-dimensional Space}\label{sec_3d}

In this section, we will prove Theorem \ref{thm_3d_new} by constructing a set of $4$ polycubes.

\subsection{Building Blocks}

Several building blocks are defined in this subsection that will be used to build the $4$ polycubes (an encoder, two linkers, and a filler) in the next subsection. A $10\times 10\times 10$ polycube is called a \textit{functional cube} (to be distinguished from a $1\times 1\times 1$ unit cube). Each building block is a subset of a functional cube. A functional cube is a trivial building block. A proper non-empty subset of a functional cube is a non-trivial building block. All non-trivial building blocks that we will define in this subsection come in pairs so that each pair can be placed together to form a functional cube exactly. All the building blocks are defined by using layer diagrams (see Figure \ref{fig_3d_u} for a layer diagram of building block $u$). Each horizontal layer is represented in the layer diagrams by a $10\times 10$ grid, where a gray square means that there is a unit cube in that position, and a white square means that the location is empty. The layers are indexed from bottom to top. In other words, the bottom layer of a building block is the first layer.

The first two pairs, $u$ and $U$, and $d$ and $D$ (see Figure \ref{fig_3d_u}, Figure \ref{fig_3d_U}, Figure \ref{fig_3d_d} and Figure \ref{fig_3d_D}, respectively), are used to encode the colors of Wang tiles.


\begin{figure}[H]
\begin{center}
\begin{tikzpicture}[scale=0.4]

\foreach \x in {0}
\foreach \y in {0} 
{
\draw [fill=gray!20] (\x+0,10+\y)--(\x+10,10+\y)--(\x+10,0+\y)--(\x+0,0+\y)--(\x+0,10+\y);
}

\foreach \x in {24}
\foreach \y in {0} 
{
\draw [fill=gray!20] (\x+0,10+\y)--(\x+10,10+\y)--(\x+10,0+\y)--(\x+0,0+\y)--(\x+0,10+\y);
\draw [fill=white!20] (\x+6,4+\y)--(\x+4,4+\y)--(\x+4,6+\y)--(\x+6,6+\y)--(\x+6,4+\y);
}

\foreach \x in {36}
\foreach \y in {0} 
{
\draw [fill=gray!20] (\x+0,10+\y)--(\x+10,10+\y)--(\x+10,0+\y)--(\x+0,0+\y)--(\x+0,10+\y);
\draw [fill=white!20] (\x+7,3+\y)--(\x+3,3+\y)--(\x+3,7+\y)--(\x+7,7+\y)--(\x+7,3+\y);
}

\foreach \x in {0,12,24,36}
\foreach \y in {0,...,10} 
{ 
\draw  (\x+0,0+\y)--(\x+10,0+\y);
\draw  (\x,0)--(\x,10);
\draw  (\x+10,0)--(\x+10,10);
\draw  (\x+9,0)--(\x+9,10);
\draw  (\x+8,0)--(\x+8,10);
\draw  (\x+7,0)--(\x+7,10);
\draw  (\x+6,0)--(\x+6,10);
\draw  (\x+5,0)--(\x+5,10);
\draw  (\x+4,0)--(\x+4,10);
\draw  (\x+3,0)--(\x+3,10);
\draw  (\x+2,0)--(\x+2,10);
\draw  (\x+1,0)--(\x+1,10);
}

\node at (5,-1) {$1$st to $5$th layers};  \node at (17,-1) {$6$th and $7$th layers};  \node at (29,-1) {$8$th layer};  \node at (41,-1) {$9$th layer}; 
\node at (5,-2) {and $10$th layer};

\end{tikzpicture}
\end{center}
\caption{Level-1 layer diagram of $u$.}\label{fig_3d_u}
\end{figure}
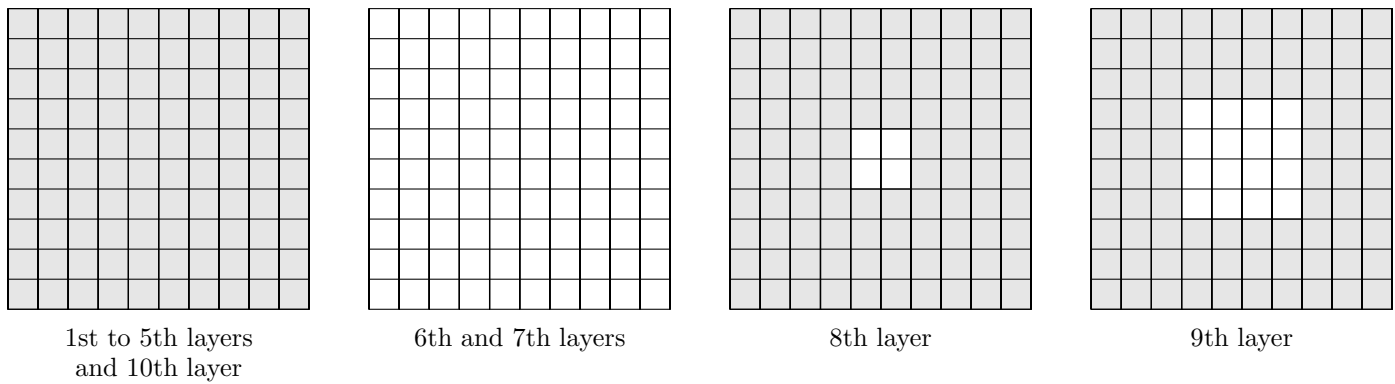


\begin{figure}[H]
\begin{center}
\begin{tikzpicture}[scale=0.4]

\foreach \x in {12}
\foreach \y in {0} 
{
\draw [fill=gray!20] (\x+0,10+\y)--(\x+10,10+\y)--(\x+10,0+\y)--(\x+0,0+\y)--(\x+0,10+\y);
}

\foreach \x in {24}
\foreach \y in {0} 
{
\draw [fill=gray!20] (\x+6,4+\y)--(\x+4,4+\y)--(\x+4,6+\y)--(\x+6,6+\y)--(\x+6,4+\y);
}

\foreach \x in {36}
\foreach \y in {0} 
{ 
\draw [fill=gray!20] (\x+7,3+\y)--(\x+3,3+\y)--(\x+3,7+\y)--(\x+7,7+\y)--(\x+7,3+\y);
}

\foreach \x in {0,12,24,36}
\foreach \y in {0,...,10} 
{ 
\draw  (\x+0,0+\y)--(\x+10,0+\y);
\draw  (\x,0)--(\x,10);
\draw  (\x+10,0)--(\x+10,10);
\draw  (\x+9,0)--(\x+9,10);
\draw  (\x+8,0)--(\x+8,10);
\draw  (\x+7,0)--(\x+7,10);
\draw  (\x+6,0)--(\x+6,10);
\draw  (\x+5,0)--(\x+5,10);
\draw  (\x+4,0)--(\x+4,10);
\draw  (\x+3,0)--(\x+3,10);
\draw  (\x+2,0)--(\x+2,10);
\draw  (\x+1,0)--(\x+1,10);
}

\node at (5,-1) {$1$st to $5$th layers};  \node at (17,-1) {$6$th and $7$th layers};  \node at (29,-1) {$8$th layer};  \node at (41,-1) {$9$th layer}; 
\node at (5,-2) {and $10$th layer};

\end{tikzpicture}
\end{center}
\caption{Level-1 layer diagram of $U$.}\label{fig_3d_U}
\end{figure}


\begin{figure}[H]
\begin{center}
\begin{tikzpicture}[scale=0.4]

\foreach \x in {36}
\foreach \y in {0} 
{
\draw [fill=gray!20] (\x+0,10+\y)--(\x+10,10+\y)--(\x+10,0+\y)--(\x+0,0+\y)--(\x+0,10+\y);
}

\foreach \x in {12}
\foreach \y in {0} 
{
\draw [fill=gray!20] (\x+0,10+\y)--(\x+10,10+\y)--(\x+10,0+\y)--(\x+0,0+\y)--(\x+0,10+\y);
\draw [fill=white!20] (\x+6,4+\y)--(\x+4,4+\y)--(\x+4,6+\y)--(\x+6,6+\y)--(\x+6,4+\y);
}

\foreach \x in {24}
\foreach \y in {0} 
{
\draw [fill=gray!20] (\x+0,10+\y)--(\x+10,10+\y)--(\x+10,0+\y)--(\x+0,0+\y)--(\x+0,10+\y);
\draw [fill=white!20] (\x+7,3+\y)--(\x+3,3+\y)--(\x+3,7+\y)--(\x+7,7+\y)--(\x+7,3+\y);
}

\foreach \x in {0,12,24,36}
\foreach \y in {0,...,10} 
{ 
\draw  (\x+0,0+\y)--(\x+10,0+\y);
\draw  (\x,0)--(\x,10);
\draw  (\x+10,0)--(\x+10,10);
\draw  (\x+9,0)--(\x+9,10);
\draw  (\x+8,0)--(\x+8,10);
\draw  (\x+7,0)--(\x+7,10);
\draw  (\x+6,0)--(\x+6,10);
\draw  (\x+5,0)--(\x+5,10);
\draw  (\x+4,0)--(\x+4,10);
\draw  (\x+3,0)--(\x+3,10);
\draw  (\x+2,0)--(\x+2,10);
\draw  (\x+1,0)--(\x+1,10);
}

\node at (5,-1) {$1$st and $2$nd layers};  \node at (17,-1) {$3$rd layer};  \node at (29,-1) {$4$th layer};  \node at (41,-1) {$5$th to $10$th layers};

\end{tikzpicture}
\end{center}
\caption{Level-1 layer diagram of $d$.}\label{fig_3d_d}
\end{figure}


\begin{figure}[H]
\begin{center}
\begin{tikzpicture}[scale=0.4]

\foreach \x in {0}
\foreach \y in {0} 
{
\draw [fill=gray!20] (\x+0,10+\y)--(\x+10,10+\y)--(\x+10,0+\y)--(\x+0,0+\y)--(\x+0,10+\y);
}

\foreach \x in {12}
\foreach \y in {0} 
{
\draw [fill=gray!20] (\x+6,4+\y)--(\x+4,4+\y)--(\x+4,6+\y)--(\x+6,6+\y)--(\x+6,4+\y);
}

\foreach \x in {24}
\foreach \y in {0} 
{ 
\draw [fill=gray!20] (\x+7,3+\y)--(\x+3,3+\y)--(\x+3,7+\y)--(\x+7,7+\y)--(\x+7,3+\y);
}

\foreach \x in {0,12,24,36}
\foreach \y in {0,...,10} 
{ 
\draw  (\x+0,0+\y)--(\x+10,0+\y);
\draw  (\x,0)--(\x,10);
\draw  (\x+10,0)--(\x+10,10);
\draw  (\x+9,0)--(\x+9,10);
\draw  (\x+8,0)--(\x+8,10);
\draw  (\x+7,0)--(\x+7,10);
\draw  (\x+6,0)--(\x+6,10);
\draw  (\x+5,0)--(\x+5,10);
\draw  (\x+4,0)--(\x+4,10);
\draw  (\x+3,0)--(\x+3,10);
\draw  (\x+2,0)--(\x+2,10);
\draw  (\x+1,0)--(\x+1,10);
}

\node at (5,-1) {$1$st and $2$nd layers};  \node at (17,-1) {$3$rd layer};  \node at (29,-1) {$4$th layer};  \node at (41,-1) {$5$th to $10$th layers};

\end{tikzpicture}
\end{center}
\caption{Level-1 layer diagram of $D$.}\label{fig_3d_D}
\end{figure}

Obviously, the building blocks $u$ are not connected. However, every building block $u$ is adjacent to another building block (mostly a $10\times 10\times 10$ functional cube) in building a bigger tile in the next subsection. Therefore, the two parts of the building block $u$ are connected through their neighboring building blocks in the bigger tile.

Note that the building blocks $U$ and $D$ are in fact identical except for their location within a $10\times 10\times 10$ functional cube. As their names indicate, $U$ lives upstairs within the functional cube and $D$ lives downstairs in the functional cube. This distinction only makes sense when $U$ and $D$ are attached to other building blocks, as we will see in the next subsection when they are used to construct the \textit{linker} tiles. When $U$ (or $D$) forms a tile by itself and can move freely without restriction within the functional cube, it becomes a \textit{filler} tile, which is one of the four tiles introduced in the next subsection.

The third pair of building blocks, $x$ and $X$, is illustrated in Figure \ref{fig_3d_x} and Figure \ref{fig_3d_X}, respectively. They will be used in the linker tile in the next subsection. Building blocks $x$ and $X$ ensure that the linkers are aligned (vertically aligned) for any two adjacent encoders in the east-west direction as we will see in the next two subsections. 


\begin{figure}[H]
\begin{center}
\begin{tikzpicture}[scale=0.4]

\foreach \x in {0}
\foreach \y in {0} 
{
\draw [fill=gray!20] (\x+0,10+\y)--(\x+10,10+\y)--(\x+10,0+\y)--(\x+0,0+\y)--(\x+0,10+\y);
}

\foreach \x in {12}
\foreach \y in {0} 
{
\draw [fill=gray!20] (\x+0,10+\y)--(\x+10,10+\y)--(\x+10,0+\y)--(\x+0,0+\y)--(\x+0,10+\y);
\draw [fill=white!20] (\x+3,3+\y)--(\x+1,3+\y)--(\x+1,7+\y)--(\x+3,7+\y)--(\x+3,3+\y);
}

\foreach \x in {24}
\foreach \y in {0} 
{
\draw [fill=gray!20] (\x+0,10+\y)--(\x+10,10+\y)--(\x+10,0+\y)--(\x+0,0+\y)--(\x+0,10+\y);
\draw [fill=white!20] (\x+3,3+\y)--(\x+1,3+\y)--(\x+1,7+\y)--(\x+3,7+\y)--(\x+3,3+\y);
\draw [fill=white!20] (\x+0,4+\y)--(\x+1,4+\y)--(\x+1,6+\y)--(\x+0,6+\y)--(\x+0,4+\y);
}

\foreach \x in {0,12,24}
\foreach \y in {0,...,10} 
{ 
\draw  (\x+0,0+\y)--(\x+10,0+\y);
\draw  (\x,0)--(\x,10);
\draw  (\x+10,0)--(\x+10,10);
\draw  (\x+9,0)--(\x+9,10);
\draw  (\x+8,0)--(\x+8,10);
\draw  (\x+7,0)--(\x+7,10);
\draw  (\x+6,0)--(\x+6,10);
\draw  (\x+5,0)--(\x+5,10);
\draw  (\x+4,0)--(\x+4,10);
\draw  (\x+3,0)--(\x+3,10);
\draw  (\x+2,0)--(\x+2,10);
\draw  (\x+1,0)--(\x+1,10);
}

\node at (5,-1) {$1$st, $2$nd, $3$rd layers};  \node at (17,-1) {$4$th and $7$th layers};  \node at (29,-1) {$5$th and $6$th layers};  
\node at (5,-2) {$8$th, $9$th, $10$th layers}; 

\end{tikzpicture}
\end{center}
\caption{Level-1 layer diagram of $x$.}\label{fig_3d_x}
\end{figure}


\begin{figure}[H]
\begin{center}
\begin{tikzpicture}[scale=0.4]

\foreach \x in {12}
\foreach \y in {0} 
{

\draw [fill=gray!20] (\x+3,3+\y)--(\x+1,3+\y)--(\x+1,7+\y)--(\x+3,7+\y)--(\x+3,3+\y);
}

\foreach \x in {24}
\foreach \y in {0} 
{

\draw [fill=gray!20] (\x+3,3+\y)--(\x+1,3+\y)--(\x+1,7+\y)--(\x+3,7+\y)--(\x+3,3+\y);
\draw [fill=gray!20] (\x+0,4+\y)--(\x+1,4+\y)--(\x+1,6+\y)--(\x+0,6+\y)--(\x+0,4+\y);
}

\foreach \x in {0,12,24}
\foreach \y in {0,...,10} 
{ 
\draw  (\x+0,0+\y)--(\x+10,0+\y);
\draw  (\x,0)--(\x,10);
\draw  (\x+10,0)--(\x+10,10);
\draw  (\x+9,0)--(\x+9,10);
\draw  (\x+8,0)--(\x+8,10);
\draw  (\x+7,0)--(\x+7,10);
\draw  (\x+6,0)--(\x+6,10);
\draw  (\x+5,0)--(\x+5,10);
\draw  (\x+4,0)--(\x+4,10);
\draw  (\x+3,0)--(\x+3,10);
\draw  (\x+2,0)--(\x+2,10);
\draw  (\x+1,0)--(\x+1,10);
}

\node at (5,-1) {$1$st, $2$nd, $3$rd layers};  \node at (17,-1) {$4$th and $7$th layers};  \node at (29,-1) {$5$th and $6$th layers};  
\node at (5,-2) {$8$th, $9$th, $10$th layers};

\end{tikzpicture}
\end{center}
\caption{Level-1 layer diagram of $X$.}\label{fig_3d_X}
\end{figure}

The fourth and fifth pairs of building blocks, $y_0$ and $Y_0$, and $y_1$ and $Y_1$ are illustrated in Figure \ref{fig_3d_y0},  Figure \ref{fig_3d_Y0}, Figure \ref{fig_3d_y1}, and Figure \ref{fig_3d_Y1}, respectively. They are also used for the purpose of vertical alignment of linkers, but they are used for two adjacent linkers in the south-north direction. 


\begin{figure}[H]
\begin{center}
\begin{tikzpicture}[scale=0.4]

\foreach \x in {0}
\foreach \y in {0} 
{
\draw [fill=gray!20] (\x+0,10+\y)--(\x+10,10+\y)--(\x+10,0+\y)--(\x+0,0+\y)--(\x+0,10+\y);
}

\foreach \x in {12}
\foreach \y in {0} 
{
\draw [fill=gray!20] (\x+0,10+\y)--(\x+10,10+\y)--(\x+10,0+\y)--(\x+0,0+\y)--(\x+0,10+\y);
\draw [fill=white!20] (\x+3,3+\y)--(\x+7,3+\y)--(\x+7,7+\y)--(\x+3,7+\y)--(\x+3,3+\y);
}

\foreach \x in {24}
\foreach \y in {0} 
{
\draw [fill=gray!20] (\x+0,10+\y)--(\x+10,10+\y)--(\x+10,0+\y)--(\x+0,0+\y)--(\x+0,10+\y);
\draw [fill=white!20] (\x+3,3+\y)--(\x+7,3+\y)--(\x+7,7+\y)--(\x+3,7+\y)--(\x+3,3+\y);
\draw [fill=white!20] (\x+4,3+\y)--(\x+6,3+\y)--(\x+6,0+\y)--(\x+4,0+\y)--(\x+4,3+\y);
}

\foreach \x in {0,12,24}
\foreach \y in {0,...,10} 
{ 
\draw  (\x+0,0+\y)--(\x+10,0+\y);
\draw  (\x,0)--(\x,10);
\draw  (\x+10,0)--(\x+10,10);
\draw  (\x+9,0)--(\x+9,10);
\draw  (\x+8,0)--(\x+8,10);
\draw  (\x+7,0)--(\x+7,10);
\draw  (\x+6,0)--(\x+6,10);
\draw  (\x+5,0)--(\x+5,10);
\draw  (\x+4,0)--(\x+4,10);
\draw  (\x+3,0)--(\x+3,10);
\draw  (\x+2,0)--(\x+2,10);
\draw  (\x+1,0)--(\x+1,10);
}

\node at (5,-1) {$1$st, $2$nd, $3$rd layers};  \node at (17,-1) {$4$th and $7$th layers};  \node at (29,-1) {$5$th and $6$th layers}; 
\node at (5,-2) {$8$th, $9$th, $10$th layers}; 

\end{tikzpicture}
\end{center}
\caption{Level-1 layer diagram of $y_0$.}\label{fig_3d_y0}
\end{figure}


\begin{figure}[H]
\begin{center}
\begin{tikzpicture}[scale=0.4]

\foreach \x in {12}
\foreach \y in {0} 
{

\draw [fill=gray!20] (\x+3,3+\y)--(\x+7,3+\y)--(\x+7,7+\y)--(\x+3,7+\y)--(\x+3,3+\y);
}

\foreach \x in {24}
\foreach \y in {0} 
{

\draw [fill=gray!20] (\x+3,3+\y)--(\x+7,3+\y)--(\x+7,7+\y)--(\x+3,7+\y)--(\x+3,3+\y);
\draw [fill=gray!20] (\x+4,3+\y)--(\x+6,3+\y)--(\x+6,0+\y)--(\x+4,0+\y)--(\x+4,3+\y);
}

\foreach \x in {0,12,24}
\foreach \y in {0,...,10} 
{ 
\draw  (\x+0,0+\y)--(\x+10,0+\y);
\draw  (\x,0)--(\x,10);
\draw  (\x+10,0)--(\x+10,10);
\draw  (\x+9,0)--(\x+9,10);
\draw  (\x+8,0)--(\x+8,10);
\draw  (\x+7,0)--(\x+7,10);
\draw  (\x+6,0)--(\x+6,10);
\draw  (\x+5,0)--(\x+5,10);
\draw  (\x+4,0)--(\x+4,10);
\draw  (\x+3,0)--(\x+3,10);
\draw  (\x+2,0)--(\x+2,10);
\draw  (\x+1,0)--(\x+1,10);
}

\node at (5,-1) {$1$st, $2$nd, $3$rd layers};  \node at (17,-1) {$4$th and $7$th layers};  \node at (29,-1) {$5$th and $6$th layers};  
\node at (5,-2) {$8$th, $9$th, $10$th layers};

\end{tikzpicture}
\end{center}
\caption{Level-1 layer diagram of $Y_0$.}\label{fig_3d_Y0}
\end{figure}


\begin{figure}[H]
\begin{center}
\begin{tikzpicture}[scale=0.4]

\foreach \x in {0}
\foreach \y in {0} 
{
\draw [fill=gray!20] (\x+0,10+\y)--(\x+10,10+\y)--(\x+10,0+\y)--(\x+0,0+\y)--(\x+0,10+\y);
}

\foreach \x in {12}
\foreach \y in {0} 
{
\draw [fill=gray!20] (\x+0,10+\y)--(\x+10,10+\y)--(\x+10,0+\y)--(\x+0,0+\y)--(\x+0,10+\y);
\draw [fill=white!20] (\x+2,4+\y)--(\x+8,4+\y)--(\x+8,6+\y)--(\x+2,6+\y)--(\x+2,4+\y);
}

\foreach \x in {24}
\foreach \y in {0} 
{
\draw [fill=gray!20] (\x+0,10+\y)--(\x+10,10+\y)--(\x+10,0+\y)--(\x+0,0+\y)--(\x+0,10+\y);
\draw [fill=white!20] (\x+2,4+\y)--(\x+8,4+\y)--(\x+8,6+\y)--(\x+2,6+\y)--(\x+2,4+\y);
\draw [fill=white!20] (\x+4,4+\y)--(\x+6,4+\y)--(\x+6,0+\y)--(\x+4,0+\y)--(\x+4,4+\y);
}

\foreach \x in {0,12,24}
\foreach \y in {0,...,10} 
{ 
\draw  (\x+0,0+\y)--(\x+10,0+\y);
\draw  (\x,0)--(\x,10);
\draw  (\x+10,0)--(\x+10,10);
\draw  (\x+9,0)--(\x+9,10);
\draw  (\x+8,0)--(\x+8,10);
\draw  (\x+7,0)--(\x+7,10);
\draw  (\x+6,0)--(\x+6,10);
\draw  (\x+5,0)--(\x+5,10);
\draw  (\x+4,0)--(\x+4,10);
\draw  (\x+3,0)--(\x+3,10);
\draw  (\x+2,0)--(\x+2,10);
\draw  (\x+1,0)--(\x+1,10);
}

\node at (5,-1) {$1$st, $2$nd, $3$rd layers};  \node at (17,-1) {$4$th and $7$th layers};  \node at (29,-1) {$5$th and $6$th layers}; 
\node at (5,-2) {$8$th, $9$th, $10$th layers}; 

\end{tikzpicture}
\end{center}
\caption{Level-1 layer diagram of $y_1$.}\label{fig_3d_y1}
\end{figure}


\begin{figure}[H]
\begin{center}
\begin{tikzpicture}[scale=0.4]

\foreach \x in {12}
\foreach \y in {0} 
{

\draw [fill=gray!20] (\x+2,4+\y)--(\x+8,4+\y)--(\x+8,6+\y)--(\x+2,6+\y)--(\x+2,4+\y);
}

\foreach \x in {24}
\foreach \y in {0} 
{

\draw [fill=gray!20] (\x+2,4+\y)--(\x+8,4+\y)--(\x+8,6+\y)--(\x+2,6+\y)--(\x+2,4+\y);
\draw [fill=gray!20] (\x+4,4+\y)--(\x+6,4+\y)--(\x+6,0+\y)--(\x+4,0+\y)--(\x+4,4+\y);
}

\foreach \x in {0,12,24}
\foreach \y in {0,...,10} 
{ 
\draw  (\x+0,0+\y)--(\x+10,0+\y);
\draw  (\x,0)--(\x,10);
\draw  (\x+10,0)--(\x+10,10);
\draw  (\x+9,0)--(\x+9,10);
\draw  (\x+8,0)--(\x+8,10);
\draw  (\x+7,0)--(\x+7,10);
\draw  (\x+6,0)--(\x+6,10);
\draw  (\x+5,0)--(\x+5,10);
\draw  (\x+4,0)--(\x+4,10);
\draw  (\x+3,0)--(\x+3,10);
\draw  (\x+2,0)--(\x+2,10);
\draw  (\x+1,0)--(\x+1,10);
}

\node at (5,-1) {$1$st, $2$nd, $3$rd layers};  \node at (17,-1) {$4$th and $7$th layers};  \node at (29,-1) {$5$th and $6$th layers};  
\node at (5,-2) {$8$th, $9$th, $10$th layers};

\end{tikzpicture}
\end{center}
\caption{Level-1 layer diagram of $Y_1$.}\label{fig_3d_Y1}
\end{figure}
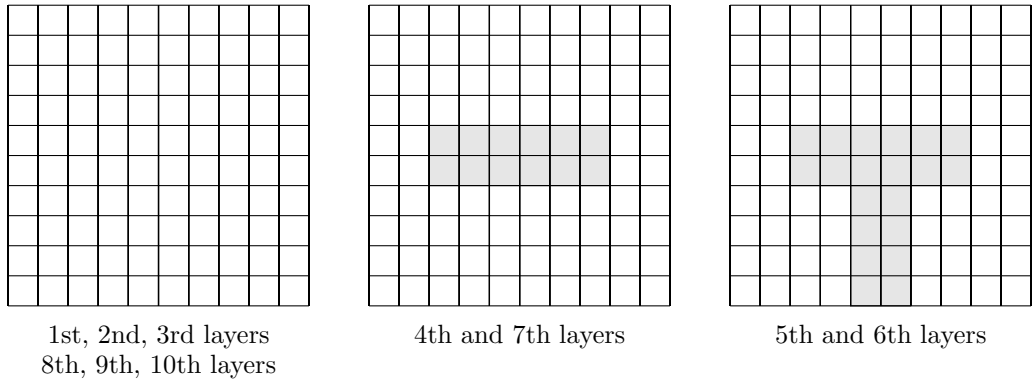

In addition to the previous two pairs, the sixth pair of building blocks, $t$ and $T$ (illustrated in Figure \ref{fig_3d_t} and Figure \ref{fig_3d_T}, respectively), also participates in ensuring vertical alignment of the linkers. The two building blocks $t$ and $T$ also help to ensure that the linkers and encoders (the encoder is a tile that will be introduced in the next subsection) must be interlocked with each other.


\begin{figure}[H]
\begin{center}
\begin{tikzpicture}[scale=0.4]

\foreach \x in {0}
\foreach \y in {0} 
{
\draw [fill=gray!20] (\x+0,10+\y)--(\x+10,10+\y)--(\x+10,0+\y)--(\x+0,0+\y)--(\x+0,10+\y);
}

\foreach \x in {12}
\foreach \y in {0} 
{
\draw [fill=gray!20] (\x+0,10+\y)--(\x+4,10+\y)--(\x+4,0+\y)--(\x+0,0+\y)--(\x+0,10+\y);
\draw [fill=gray!20] (\x+10,10+\y)--(\x+6,10+\y)--(\x+6,0+\y)--(\x+10,0+\y)--(\x+10,10+\y);

\draw [fill=white] (\x+4,7+\y)--(\x+2,7+\y)--(\x+2,3+\y)--(\x+4,3+\y)--(\x+4,7+\y);
\draw [fill=white] (\x+6,7+\y)--(\x+8,7+\y)--(\x+8,3+\y)--(\x+6,3+\y)--(\x+6,7+\y);
\draw [fill=gray!20] (\x+4,6+\y)--(\x+3,6+\y)--(\x+3,4+\y)--(\x+4,4+\y)--(\x+4,6+\y);
\draw [fill=gray!20] (\x+6,6+\y)--(\x+7,6+\y)--(\x+7,4+\y)--(\x+6,4+\y)--(\x+6,6+\y);

}

\foreach \x in {0,12}
\foreach \y in {0,...,10} 
{ 
\draw  (\x+0,0+\y)--(\x+10,0+\y);
\draw  (\x,0)--(\x,10);
\draw  (\x+10,0)--(\x+10,10);
\draw  (\x+9,0)--(\x+9,10);
\draw  (\x+8,0)--(\x+8,10);
\draw  (\x+7,0)--(\x+7,10);
\draw  (\x+6,0)--(\x+6,10);
\draw  (\x+5,0)--(\x+5,10);
\draw  (\x+4,0)--(\x+4,10);
\draw  (\x+3,0)--(\x+3,10);
\draw  (\x+2,0)--(\x+2,10);
\draw  (\x+1,0)--(\x+1,10);
}

\node at (5,-1) {$1$st to $4$th layers};  \node at (17,-1) {$5$th and $6$th layers};  
\node at (5,-2) {and $7$th to $10$th layers}; 

\end{tikzpicture}
\end{center}
\caption{Level-1 layer diagram of $t$.}\label{fig_3d_t}
\end{figure}


\begin{figure}[H]
\begin{center}
\begin{tikzpicture}[scale=0.4]

\foreach \x in {12}
\foreach \y in {0} 
{

\draw [fill=gray!20] (\x+6,10+\y)--(\x+4,10+\y)--(\x+4,0+\y)--(\x+6,0+\y)--(\x+6,10+\y);

\draw [fill=gray!20] (\x+4,7+\y)--(\x+2,7+\y)--(\x+2,3+\y)--(\x+4,3+\y)--(\x+4,7+\y);
\draw [fill=gray!20] (\x+6,7+\y)--(\x+8,7+\y)--(\x+8,3+\y)--(\x+6,3+\y)--(\x+6,7+\y);
\draw [fill=white] (\x+4,6+\y)--(\x+3,6+\y)--(\x+3,4+\y)--(\x+4,4+\y)--(\x+4,6+\y);
\draw [fill=white] (\x+6,6+\y)--(\x+7,6+\y)--(\x+7,4+\y)--(\x+6,4+\y)--(\x+6,6+\y);

}

\foreach \x in {0,12}
\foreach \y in {0,...,10} 
{ 
\draw  (\x+0,0+\y)--(\x+10,0+\y);
\draw  (\x,0)--(\x,10);
\draw  (\x+10,0)--(\x+10,10);
\draw  (\x+9,0)--(\x+9,10);
\draw  (\x+8,0)--(\x+8,10);
\draw  (\x+7,0)--(\x+7,10);
\draw  (\x+6,0)--(\x+6,10);
\draw  (\x+5,0)--(\x+5,10);
\draw  (\x+4,0)--(\x+4,10);
\draw  (\x+3,0)--(\x+3,10);
\draw  (\x+2,0)--(\x+2,10);
\draw  (\x+1,0)--(\x+1,10);
}

\node at (5,-1) {$1$st to $4$th layers};  \node at (17,-1) {$5$th and $6$th layers};  
\node at (5,-2) {and $7$th to $10$th layers};

\end{tikzpicture}
\end{center}
\caption{Level-1 layer diagram of $T$.}\label{fig_3d_T}
\end{figure}

The last two pairs of building blocks, $e$ and $\mathbb{E}$, and $l$ and $\mathbb{L}$ (illustrated in Figure \ref{fig_3d_e}, Figure \ref{fig_3d_E}, Figure \ref{fig_3d_l} and Figure \ref{fig_3d_L}, respectively), are used to ensure that an encoder must be placed on top of another encoder and a linker must be placed on top of another linker, respectively.


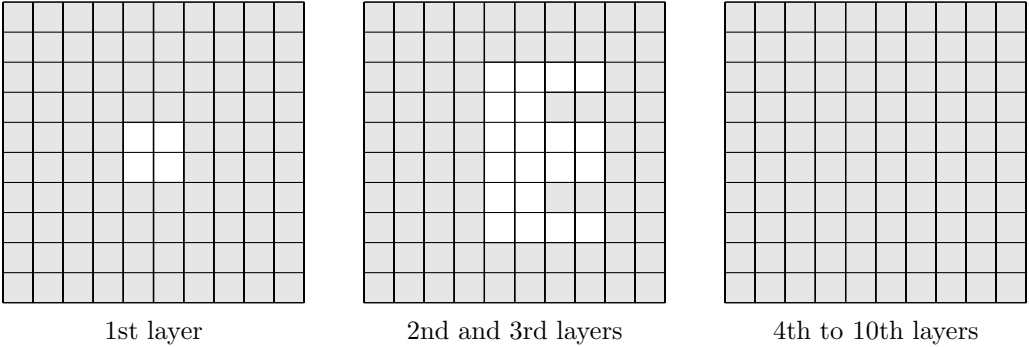
\begin{figure}[H]
\begin{center}
\begin{tikzpicture}[scale=0.4]

\foreach \x in {24}
\foreach \y in {0} 
{
\draw [fill=gray!20] (\x+0,10+\y)--(\x+10,10+\y)--(\x+10,0+\y)--(\x+0,0+\y)--(\x+0,10+\y);
}

\foreach \x in {0}
\foreach \y in {0} 
{
\draw [fill=gray!20] (\x+0,10+\y)--(\x+10,10+\y)--(\x+10,0+\y)--(\x+0,0+\y)--(\x+0,10+\y);
\draw [fill=white!20] (\x+6,4+\y)--(\x+4,4+\y)--(\x+4,6+\y)--(\x+6,6+\y)--(\x+6,4+\y);
}

\foreach \x in {12}
\foreach \y in {0} 
{
\draw [fill=gray!20] (\x+0,10+\y)--(\x+10,10+\y)--(\x+10,0+\y)--(\x+0,0+\y)--(\x+0,10+\y);
\draw [fill=white!20] (\x+8,4+\y)--(\x+6,4+\y)--(\x+6,6+\y)--(\x+8,6+\y)--(\x+8,4+\y);
\draw [fill=white!20] (\x+8,7+\y)--(\x+6,7+\y)--(\x+6,8+\y)--(\x+8,8+\y)--(\x+8,7+\y);
\draw [fill=white!20] (\x+8,2+\y)--(\x+6,2+\y)--(\x+6,3+\y)--(\x+8,3+\y)--(\x+8,2+\y);
\draw [fill=white!20] (\x+6,2+\y)--(\x+4,2+\y)--(\x+4,8+\y)--(\x+6,8+\y)--(\x+6,2+\y);
}

\foreach \x in {0,12,24}
\foreach \y in {0,...,10} 
{ 
\draw  (\x+0,0+\y)--(\x+10,0+\y);
\draw  (\x,0)--(\x,10);
\draw  (\x+10,0)--(\x+10,10);
\draw  (\x+9,0)--(\x+9,10);
\draw  (\x+8,0)--(\x+8,10);
\draw  (\x+7,0)--(\x+7,10);
\draw  (\x+6,0)--(\x+6,10);
\draw  (\x+5,0)--(\x+5,10);
\draw  (\x+4,0)--(\x+4,10);
\draw  (\x+3,0)--(\x+3,10);
\draw  (\x+2,0)--(\x+2,10);
\draw  (\x+1,0)--(\x+1,10);
}

\node at (5,-1) {$1$st layer};  \node at (17,-1) {$2$nd and $3$rd layers};  \node at (29,-1) {$4$th to $10$th layers};

\end{tikzpicture}
\end{center}
\caption{Level-1 layer diagram of $e$.}\label{fig_3d_e}
\end{figure}


\begin{figure}[H]
\begin{center}
\begin{tikzpicture}[scale=0.4]

\foreach \x in {0}
\foreach \y in {0} 
{

\draw [fill=gray!20] (\x+6,4+\y)--(\x+4,4+\y)--(\x+4,6+\y)--(\x+6,6+\y)--(\x+6,4+\y);
}

\foreach \x in {12}
\foreach \y in {0} 
{

\draw [fill=gray!20] (\x+8,4+\y)--(\x+6,4+\y)--(\x+6,6+\y)--(\x+8,6+\y)--(\x+8,4+\y);
\draw [fill=gray!20] (\x+8,7+\y)--(\x+6,7+\y)--(\x+6,8+\y)--(\x+8,8+\y)--(\x+8,7+\y);
\draw [fill=gray!20] (\x+8,2+\y)--(\x+6,2+\y)--(\x+6,3+\y)--(\x+8,3+\y)--(\x+8,2+\y);
\draw [fill=gray!20] (\x+6,2+\y)--(\x+4,2+\y)--(\x+4,8+\y)--(\x+6,8+\y)--(\x+6,2+\y);
}

\foreach \x in {0,12,24}
\foreach \y in {0,...,10} 
{ 
\draw  (\x+0,0+\y)--(\x+10,0+\y);
\draw  (\x,0)--(\x,10);
\draw  (\x+10,0)--(\x+10,10);
\draw  (\x+9,0)--(\x+9,10);
\draw  (\x+8,0)--(\x+8,10);
\draw  (\x+7,0)--(\x+7,10);
\draw  (\x+6,0)--(\x+6,10);
\draw  (\x+5,0)--(\x+5,10);
\draw  (\x+4,0)--(\x+4,10);
\draw  (\x+3,0)--(\x+3,10);
\draw  (\x+2,0)--(\x+2,10);
\draw  (\x+1,0)--(\x+1,10);
}

\node at (5,-1) {$1$st layer};  \node at (17,-1) {$2$nd and $3$rd layers};  \node at (29,-1) {$4$th to $10$th layers};

\end{tikzpicture}
\end{center}
\caption{Level-1 layer diagram of $\mathbb{E}$.}\label{fig_3d_E}
\end{figure}
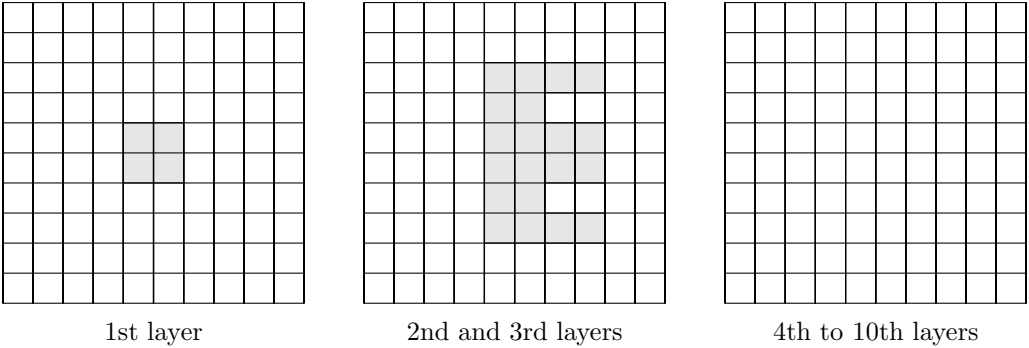


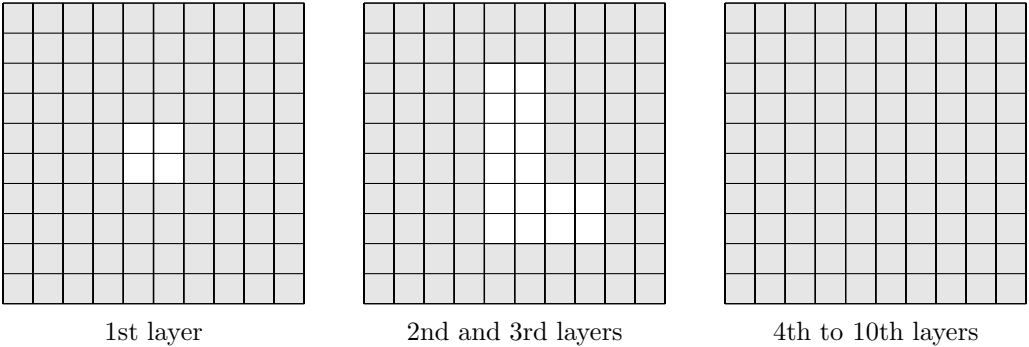
\begin{figure}[H]
\begin{center}
\begin{tikzpicture}[scale=0.4]

\foreach \x in {24}
\foreach \y in {0} 
{
\draw [fill=gray!20] (\x+0,10+\y)--(\x+10,10+\y)--(\x+10,0+\y)--(\x+0,0+\y)--(\x+0,10+\y);
}

\foreach \x in {0}
\foreach \y in {0} 
{
\draw [fill=gray!20] (\x+0,10+\y)--(\x+10,10+\y)--(\x+10,0+\y)--(\x+0,0+\y)--(\x+0,10+\y);
\draw [fill=white!20] (\x+6,4+\y)--(\x+4,4+\y)--(\x+4,6+\y)--(\x+6,6+\y)--(\x+6,4+\y);
}

\foreach \x in {12}
\foreach \y in {0} 
{
\draw [fill=gray!20] (\x+0,10+\y)--(\x+10,10+\y)--(\x+10,0+\y)--(\x+0,0+\y)--(\x+0,10+\y);

\draw [fill=white!20] (\x+8,2+\y)--(\x+6,2+\y)--(\x+6,4+\y)--(\x+8,4+\y)--(\x+8,2+\y);
\draw [fill=white!20] (\x+6,2+\y)--(\x+4,2+\y)--(\x+4,8+\y)--(\x+6,8+\y)--(\x+6,2+\y);
}

\foreach \x in {0,12,24}
\foreach \y in {0,...,10} 
{ 
\draw  (\x+0,0+\y)--(\x+10,0+\y);
\draw  (\x,0)--(\x,10);
\draw  (\x+10,0)--(\x+10,10);
\draw  (\x+9,0)--(\x+9,10);
\draw  (\x+8,0)--(\x+8,10);
\draw  (\x+7,0)--(\x+7,10);
\draw  (\x+6,0)--(\x+6,10);
\draw  (\x+5,0)--(\x+5,10);
\draw  (\x+4,0)--(\x+4,10);
\draw  (\x+3,0)--(\x+3,10);
\draw  (\x+2,0)--(\x+2,10);
\draw  (\x+1,0)--(\x+1,10);
}

\node at (5,-1) {$1$st layer};  \node at (17,-1) {$2$nd and $3$rd layers};  \node at (29,-1) {$4$th to $10$th layers};

\end{tikzpicture}
\end{center}
\caption{Level-1 layer diagram of $l$.}\label{fig_3d_l}
\end{figure}


\begin{figure}[H]
\begin{center}
\begin{tikzpicture}[scale=0.4]

\foreach \x in {0}
\foreach \y in {0} 
{

\draw [fill=gray!20] (\x+6,4+\y)--(\x+4,4+\y)--(\x+4,6+\y)--(\x+6,6+\y)--(\x+6,4+\y);
}

\foreach \x in {12}
\foreach \y in {0} 
{

\draw [fill=gray!20] (\x+8,2+\y)--(\x+6,2+\y)--(\x+6,4+\y)--(\x+8,4+\y)--(\x+8,2+\y);
\draw [fill=gray!20] (\x+6,2+\y)--(\x+4,2+\y)--(\x+4,8+\y)--(\x+6,8+\y)--(\x+6,2+\y);
}

\foreach \x in {0,12,24}
\foreach \y in {0,...,10} 
{ 
\draw  (\x+0,0+\y)--(\x+10,0+\y);
\draw  (\x,0)--(\x,10);
\draw  (\x+10,0)--(\x+10,10);
\draw  (\x+9,0)--(\x+9,10);
\draw  (\x+8,0)--(\x+8,10);
\draw  (\x+7,0)--(\x+7,10);
\draw  (\x+6,0)--(\x+6,10);
\draw  (\x+5,0)--(\x+5,10);
\draw  (\x+4,0)--(\x+4,10);
\draw  (\x+3,0)--(\x+3,10);
\draw  (\x+2,0)--(\x+2,10);
\draw  (\x+1,0)--(\x+1,10);
}

\node at (5,-1) {$1$st layer};  \node at (17,-1) {$2$nd and $3$rd layers};  \node at (29,-1) {$4$th to $10$th layers};

\end{tikzpicture}
\end{center}
\caption{Level-1 layer diagram of $\mathbb{L}$.}\label{fig_3d_L}
\end{figure}
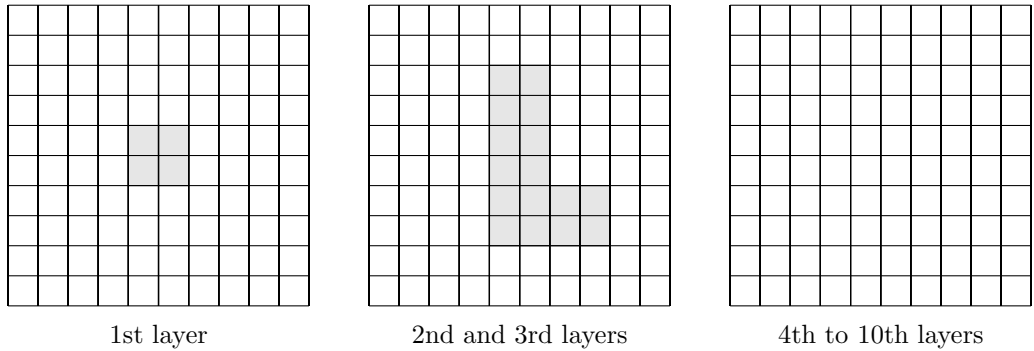

It is easy to check that all the building blocks are connected except $u$. As we mentioned earlier, the building blocks $u$ will be connected through their neighbors. So all the tiles constructed in the next subsection are connected.

We make several remarks to conclude this subsection. Each square of the layer diagrams in the figures in this subsection represents a $1\times 1\times 1$ unit cube, so they are called \textit{level}-1 layer diagram, which is to be distinguished from the \textit{level}-2 diagram in the next subsection, where a square represents a building block. Note also that the size and shape of the building blocks we have defined in this subsection are kind of arbitrary. They can be replaced by building blocks of other sizes and shapes as long as each pair is a perfect match to each other (forming a $10\times 10\times 10$ functional cube), but different pairs cannot be matched. Also, we do not try to minimize the size of the building blocks in our proof of Theorem \ref{thm_3d_new}. Our proof of Theorem~\ref{thm_3d_new} can likely be revised to use building blocks based on polycubes of size $8\times 8\times 8$ or even smaller.

\subsection{The Set of Four Tiles}

Given a set of Wang tiles, we will construct a set of $4$ polycubes that will simulate the tilings of the set of Wang tiles. We take the set of $3$ Wang tiles in Figure \ref{fig_w3} as an example to describe the method of construction of the set of $4$ polycubes. The construction method can be easily generalized to any set of Wang tiles. 

There are $4$ different colors in the set of Wang tiles in Figure \ref{fig_w3}. Let the four colors red, green, blue, and yellow be encoded by binary strings 00, 01, 10 and 11, respectively. We add two redundant bits to the strings, a prefix $0$ and a suffix $1$. Therefore, the complete binary codes for the four colors red, green, blue, and yellow are 0001, 0011, 0101 and 0111. So it is obvious that 0000 and 1111 are not used in our code words for the colors. In general, let $q$ be the number of different colors in a set of Wang tiles, then the colors can be encoded in the same way by using binary strings of length $t=\lceil \log_2 q\rceil + 2$. The binary strings for encoding the colors start with $0$ and end with $1$, and the colors are actually encoded by the $\lceil \log_2 q\rceil$ bits in the middle. This allows us to use the two binary strings: $t$ consecutive $0$s and $t$ consecutive $1$s for other purposes. Physically on a polycube, $0$ and $1$ will be represented by building blocks $d$ and $u$, respectively.

With the above encoding method in hand, we can introduce the \textit{encoder} tile as illustrated in Figure \ref{fig_3d_encoder}. As we have mentioned in the previous subsection, each square in the layer diagrams in this subsection represents a $10\times 10\times 10$ functional cube or other building blocks, so they are called \textit{level}-2 layer diagram. As depicted in Figure \ref{fig_3d_encoder}, the encoder is roughly (by ignoring the dents and bumps on some of the building blocks) a cuboid of $16\times 3\times 6$ building blocks. In other words, the main body of the encoder consists of $6$ layers (the seventh layer is regarded as a small bump that is attached to the main body) of building blocks, and each layer consists of a $16\times 3$ rectangle of building blocks.

The first layer, third layer and fifth layer are the encoding layers, and each layer simulates a Wang tile. The simulated Wang tile is depicted on the right of each layer for easier comparison. A Wang tile is simulated in an encoding layer in the following way. Both the north side and south side of an encoding layer are separated into four segments. Each segment consists of four consecutive building blocks. On the north side, the first segment is a sequence of four building blocks $uuuu$, the second and third segments encode the colors of two sides of a Wang tile using the encoding method we just introduced, and the last segment is $dddd$. On the south side, the first and last segments encode the colors of the other two sides of the same Wang tile, and the two central segments are $dddd$ and $uuuu$, respectively. The presence of the non-color-encoding segments ($dddd$ and $uuuu$) plays an important role in deciding the relative position between the encoders in a tiling as we will see in the next subsection.


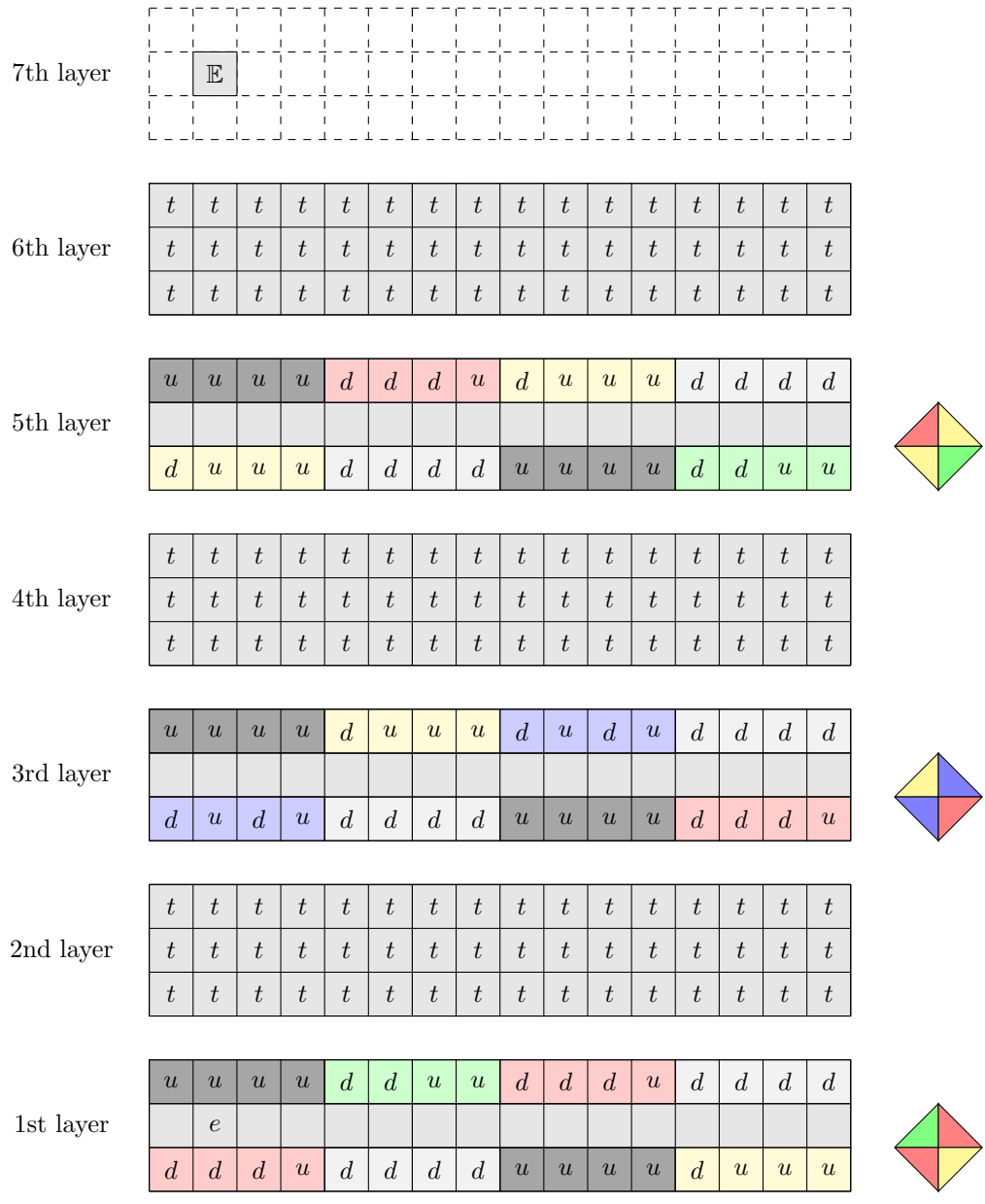
\begin{figure}[H]
\begin{center}
\begin{tikzpicture}[scale=0.6,pattern1/.style={draw=red,pattern color=gray!70, pattern=north east lines}]

\foreach \x in {0}
\foreach \y in {2,10,18} 
{
\draw [ fill=gray!70] (\x+0,\y+0)--(\x+4,\y+0)--(\x+4,\y+1)--(\x+0,\y+1)--(\x+0,\y+0);
\node at (\x+0.5,\y+0.5) {$u$};  
\node at (\x+1.5,\y+0.5) {$u$};  
\node at (\x+2.5,\y+0.5) {$u$};  
\node at (\x+3.5,\y+0.5) {$u$};  
}
\foreach \x in {8}
\foreach \y in {0,8,16} 
{
\draw [ fill=gray!70] (\x+0,\y+0)--(\x+4,\y+0)--(\x+4,\y+1)--(\x+0,\y+1)--(\x+0,\y+0);
\node at (\x+0.5,\y+0.5) {$u$};  
\node at (\x+1.5,\y+0.5) {$u$};  
\node at (\x+2.5,\y+0.5) {$u$};  
\node at (\x+3.5,\y+0.5) {$u$};  
}
\foreach \x in {4}
\foreach \y in {0,8,16} 
{
\draw [ fill=gray!10 ] (\x+0,\y+0)--(\x+4,\y+0)--(\x+4,\y+1)--(\x+0,\y+1)--(\x+0,\y+0);
\node at (\x+0.5,\y+0.5) {$d$};  
\node at (\x+1.5,\y+0.5) {$d$};  
\node at (\x+2.5,\y+0.5) {$d$};  
\node at (\x+3.5,\y+0.5) {$d$};  
}
\foreach \x in {12}
\foreach \y in {2,10,18} 
{
\draw [ fill=gray!10 ] (\x+0,\y+0)--(\x+4,\y+0)--(\x+4,\y+1)--(\x+0,\y+1)--(\x+0,\y+0);
\node at (\x+0.5,\y+0.5) {$d$};  
\node at (\x+1.5,\y+0.5) {$d$};  
\node at (\x+2.5,\y+0.5) {$d$};  
\node at (\x+3.5,\y+0.5) {$d$};  
}

\foreach \x in {0}
\foreach \y in {1,9,17} 
{
\draw [ fill=gray!20] (\x+0,\y+0)--(\x+16,\y+0)--(\x+16,\y+1)--(\x+0,\y+1)--(\x+0,\y+0);
}
\foreach \x in {0}
\foreach \y in {4,12,20} 
{
\draw [ fill=gray!20] (\x+0,\y+0)--(\x+16,\y+0)--(\x+16,\y+3)--(\x+0,\y+3)--(\x+0,\y+0);
}


\foreach \x in {0}
\foreach \y in {0} 
{
\draw [ fill=red!20] (\x+0,\y+0)--(\x+4,\y+0)--(\x+4,\y+1)--(\x+0,\y+1)--(\x+0,\y+0);
\node at (\x+0.5,\y+0.5) {$d$};  
\node at (\x+1.5,\y+0.5) {$d$};  
\node at (\x+2.5,\y+0.5) {$d$};  
\node at (\x+3.5,\y+0.5) {$u$};  
}
\foreach \x in {8}
\foreach \y in {2} 
{
\draw [ fill=red!20] (\x+0,\y+0)--(\x+4,\y+0)--(\x+4,\y+1)--(\x+0,\y+1)--(\x+0,\y+0);
\node at (\x+0.5,\y+0.5) {$d$};  
\node at (\x+1.5,\y+0.5) {$d$};  
\node at (\x+2.5,\y+0.5) {$d$};  
\node at (\x+3.5,\y+0.5) {$u$};  
}

\foreach \x in {12}
\foreach \y in {8} 
{
\draw [ fill=red!20] (\x+0,\y+0)--(\x+4,\y+0)--(\x+4,\y+1)--(\x+0,\y+1)--(\x+0,\y+0);
\node at (\x+0.5,\y+0.5) {$d$};  
\node at (\x+1.5,\y+0.5) {$d$};  
\node at (\x+2.5,\y+0.5) {$d$};  
\node at (\x+3.5,\y+0.5) {$u$};  
}

\foreach \x in {4}
\foreach \y in {18} 
{
\draw [ fill=red!20] (\x+0,\y+0)--(\x+4,\y+0)--(\x+4,\y+1)--(\x+0,\y+1)--(\x+0,\y+0);
\node at (\x+0.5,\y+0.5) {$d$};  
\node at (\x+1.5,\y+0.5) {$d$};  
\node at (\x+2.5,\y+0.5) {$d$};  
\node at (\x+3.5,\y+0.5) {$u$};  
}


\foreach \x in {12}
\foreach \y in {0} 
{
\draw [ fill=yellow!20] (\x+0,\y+0)--(\x+4,\y+0)--(\x+4,\y+1)--(\x+0,\y+1)--(\x+0,\y+0);
\node at (\x+0.5,\y+0.5) {$d$};  
\node at (\x+1.5,\y+0.5) {$u$};  
\node at (\x+2.5,\y+0.5) {$u$};  
\node at (\x+3.5,\y+0.5) {$u$};  
}
\foreach \x in {4}
\foreach \y in {10} 
{
\draw [ fill=yellow!20] (\x+0,\y+0)--(\x+4,\y+0)--(\x+4,\y+1)--(\x+0,\y+1)--(\x+0,\y+0);
\node at (\x+0.5,\y+0.5) {$d$};  
\node at (\x+1.5,\y+0.5) {$u$};  
\node at (\x+2.5,\y+0.5) {$u$};  
\node at (\x+3.5,\y+0.5) {$u$};  
}

\foreach \x in {0}
\foreach \y in {16} 
{
\draw [ fill=yellow!20] (\x+0,\y+0)--(\x+4,\y+0)--(\x+4,\y+1)--(\x+0,\y+1)--(\x+0,\y+0);
\node at (\x+0.5,\y+0.5) {$d$};  
\node at (\x+1.5,\y+0.5) {$u$};  
\node at (\x+2.5,\y+0.5) {$u$};  
\node at (\x+3.5,\y+0.5) {$u$};  
}

\foreach \x in {8}
\foreach \y in {18} 
{
\draw [ fill=yellow!20] (\x+0,\y+0)--(\x+4,\y+0)--(\x+4,\y+1)--(\x+0,\y+1)--(\x+0,\y+0);
\node at (\x+0.5,\y+0.5) {$d$};  
\node at (\x+1.5,\y+0.5) {$u$};  
\node at (\x+2.5,\y+0.5) {$u$};  
\node at (\x+3.5,\y+0.5) {$u$};  
}

\foreach \x in {4}
\foreach \y in {2} 
{
\draw [ fill=green!20] (\x+0,\y+0)--(\x+4,\y+0)--(\x+4,\y+1)--(\x+0,\y+1)--(\x+0,\y+0);
\node at (\x+0.5,\y+0.5) {$d$};  
\node at (\x+1.5,\y+0.5) {$d$};  
\node at (\x+2.5,\y+0.5) {$u$};  
\node at (\x+3.5,\y+0.5) {$u$};  
}

\foreach \x in {12}
\foreach \y in {16} 
{
\draw [ fill=green!20] (\x+0,\y+0)--(\x+4,\y+0)--(\x+4,\y+1)--(\x+0,\y+1)--(\x+0,\y+0);
\node at (\x+0.5,\y+0.5) {$d$};  
\node at (\x+1.5,\y+0.5) {$d$};  
\node at (\x+2.5,\y+0.5) {$u$};  
\node at (\x+3.5,\y+0.5) {$u$};  
}

\foreach \x in {8}
\foreach \y in {10} 
{
\draw [ fill=blue!20] (\x+0,\y+0)--(\x+4,\y+0)--(\x+4,\y+1)--(\x+0,\y+1)--(\x+0,\y+0);
\node at (\x+0.5,\y+0.5) {$d$};  
\node at (\x+1.5,\y+0.5) {$u$};  
\node at (\x+2.5,\y+0.5) {$d$};  
\node at (\x+3.5,\y+0.5) {$u$};  
}
\foreach \x in {0}
\foreach \y in {8} 
{
\draw [ fill=blue!20] (\x+0,\y+0)--(\x+4,\y+0)--(\x+4,\y+1)--(\x+0,\y+1)--(\x+0,\y+0);
\node at (\x+0.5,\y+0.5) {$d$};  
\node at (\x+1.5,\y+0.5) {$u$};  
\node at (\x+2.5,\y+0.5) {$d$};  
\node at (\x+3.5,\y+0.5) {$u$};  
}


\foreach \x in {0}
\foreach \y in {0,...,23} 
{ 
\draw  (\x+0,0+\y)--(\x+16,0+\y);
}

\foreach \x in {0}
\foreach \y in {0,4,8,12,16,20} 
{ 
\draw  (\x,0+\y)--(\x,3+\y);
\draw  (\x+10,0+\y)--(\x+10,3+\y);
\draw  (\x+9,0+\y)--(\x+9,3+\y);
\draw  (\x+8,0+\y)--(\x+8,3+\y);
\draw  (\x+7,0+\y)--(\x+7,3+\y);
\draw  (\x+6,0+\y)--(\x+6,3+\y);
\draw  (\x+5,0+\y)--(\x+5,3+\y);
\draw  (\x+4,0+\y)--(\x+4,3+\y);
\draw  (\x+3,0+\y)--(\x+3,3+\y);
\draw  (\x+2,0+\y)--(\x+2,3+\y);
\draw  (\x+1,0+\y)--(\x+1,3+\y);
\draw  (\x+6,0+\y)--(\x+6,3+\y);
\draw  (\x+15,0+\y)--(\x+15,3+\y);
\draw  (\x+14,0+\y)--(\x+14,3+\y);
\draw  (\x+13,0+\y)--(\x+13,3+\y);
\draw  (\x+12,0+\y)--(\x+12,3+\y);
\draw  (\x+11,0+\y)--(\x+11,3+\y);
\draw  (\x+16,0+\y)--(\x+16,3+\y);
}

\foreach \x in {0}
\foreach \y in {24} 
{ 
\draw [dashed]  (0,\y)--(16,\y);\draw [dashed]  (0,\y+3)--(16,\y+3);
\draw [dashed]  (0,\y+2)--(16,\y+2); \draw [dashed]  (0,\y+1)--(16,\y+1);

\draw [dashed]  (\x,0+\y)--(\x,3+\y);
\draw [dashed]  (\x+10,0+\y)--(\x+10,3+\y);
\draw [dashed]  (\x+9,0+\y)--(\x+9,3+\y);
\draw [dashed]  (\x+8,0+\y)--(\x+8,3+\y);
\draw [dashed]  (\x+7,0+\y)--(\x+7,3+\y);
\draw [dashed]  (\x+6,0+\y)--(\x+6,3+\y);
\draw [dashed]  (\x+5,0+\y)--(\x+5,3+\y);
\draw [dashed]  (\x+4,0+\y)--(\x+4,3+\y);
\draw [dashed]  (\x+3,0+\y)--(\x+3,3+\y);
\draw [dashed]  (\x+2,0+\y)--(\x+2,3+\y);
\draw [dashed]  (\x+1,0+\y)--(\x+1,3+\y);
\draw [dashed]  (\x+6,0+\y)--(\x+6,3+\y);
\draw [dashed]  (\x+15,0+\y)--(\x+15,3+\y);
\draw [dashed]  (\x+14,0+\y)--(\x+14,3+\y);
\draw [dashed]  (\x+13,0+\y)--(\x+13,3+\y);
\draw [dashed]  (\x+12,0+\y)--(\x+12,3+\y);
\draw [dashed]  (\x+11,0+\y)--(\x+11,3+\y);
\draw [dashed]  (\x+16,0+\y)--(\x+16,3+\y);
}

\node at (-2,1.5) {$1$st layer};  
\node at (-2,5.5) {$2$nd layer};  
\node at (-2,9.5) {$3$rd layer};  
\node at (-2,13.5) {$4$th layer};    
\node at (-2,17.5) {$5$th layer};  
\node at (-2,21.5) {$6$th layer};
\node at (-2,25.5) {$7$th layer};

\foreach \x in {17}
\foreach \y in {1} 
{ 
\draw [ fill=green!50] (\x+0,0+\y)--(\x+1,1+\y)--(\x+1,0+\y)--(\x+0,0+\y);
\draw [ fill=red!50] (\x+0,0+\y)--(\x+1,-1+\y)--(\x+1,0+\y)--(\x+0,0+\y);
\draw [ fill=red!50] (\x+2,0+\y)--(\x+1,1+\y)--(\x+1,0+\y)--(\x+2,0+\y);
\draw [ fill=yellow!50] (\x+2,0+\y)--(\x+1,-1+\y)--(\x+1,0+\y)--(\x+2,0+\y);
}

\foreach \x in {17}
\foreach \y in {9} 
{ 
\draw [ fill=yellow!50] (\x+0,0+\y)--(\x+1,1+\y)--(\x+1,0+\y)--(\x+0,0+\y);
\draw [ fill=blue!50] (\x+0,0+\y)--(\x+1,-1+\y)--(\x+1,0+\y)--(\x+0,0+\y);
\draw [ fill=blue!50] (\x+2,0+\y)--(\x+1,1+\y)--(\x+1,0+\y)--(\x+2,0+\y);
\draw [ fill=red!50] (\x+2,0+\y)--(\x+1,-1+\y)--(\x+1,0+\y)--(\x+2,0+\y);
}

\foreach \x in {17}
\foreach \y in {17} 
{ 
\draw [ fill=red!50] (\x+0,0+\y)--(\x+1,1+\y)--(\x+1,0+\y)--(\x+0,0+\y);
\draw [ fill=yellow!50] (\x+0,0+\y)--(\x+1,-1+\y)--(\x+1,0+\y)--(\x+0,0+\y);
\draw [ fill=yellow!50] (\x+2,0+\y)--(\x+1,1+\y)--(\x+1,0+\y)--(\x+2,0+\y);
\draw [ fill=green!50] (\x+2,0+\y)--(\x+1,-1+\y)--(\x+1,0+\y)--(\x+2,0+\y);
}

\foreach \x in {0,...,15}
\foreach \y in {4,5,6,12,13,14,20,21,22} 
{ 
\node at (\x+0.5,\y+0.5) {$t$};  
}

\foreach \x in {1}
\foreach \y in {25} 
{
\draw [ fill=gray!20] (\x+0,\y+0)--(\x+1,\y+0)--(\x+1,\y+1)--(\x+0,\y+1)--(\x+0,\y+0);
\node at (\x+0.5,\y+0.5) {$\mathbb{E}$};  
}

\foreach \x in {1}
\foreach \y in {1} 
{
\node at (\x+0.5,\y+0.5) {$e$};  
}

\end{tikzpicture}
\end{center}
\caption{Level-2 layer diagram the encoder.}\label{fig_3d_encoder}
\end{figure}

The second, fourth and sixth layers are non-encoding layers. Each of these non-encoding layers is a $16\times 3$ rectangle of building blocks $t$. These non-encoding layers will help to ensure the tiling pattern as we will explain later in the next subsection. By the definition of building blocks $t$, we can see easily that each of these non-encoding layers contains $16$ parallel tunnels in the south-north direction.

In general, for a set of $p$ Wang tiles with $q$ colors (let $t=\lceil \log_2 q\rceil + 2$ as before), the corresponding encoder tile is (roughly) a cuboid of $4t\times 3\times 2p$ building blocks. The $(2i+1)$-th ($i=0,1,\cdots, p-1$) layers are called the \textit{encoding} layers, as each encoding layer encodes a Wang tile. All the other layers are \textit{non-encoding} layers. Therefore, the encoder tile has $4pt$ tunnels in the non-encoding layers. So the genus of the surface of the encoder tile increases as the number of Wang tiles and the number of different colors in the set of Wang tiles increases.

There are two more non-trivial building blocks on the encoder tile. One is the building block $e$ on the first layer, and the other is the building block $\mathbb{E}$ on the seventh layer. By the definition of the building block $\mathbb{E}$, it is easy to check that it is connected to the sixth layer of the encoder. Therefore, the encoder tile is a connected polycube. The building blocks $e$ and $\mathbb{E}$ ensure that we must place another encoder on top or below any encoder. We state this property of the encoders as the following fact for later reference.

\begin{Fact}\label{fct_column}
    In any tilings, the encoder tiles must form vertical two-way infinite arrays.
\end{Fact}

All other gray squares without labels in Figure \ref{fig_3d_encoder} represent trivial functional cubes without dents or bumps.

There are two \textit{linker} polycubes. One is illustrated in Figure \ref{fig_3d_linker}, which is called the $U$-linker. If we replace the two building blocks $U$ by $D$ in the first layer of the linker in Figure \ref{fig_3d_linker}, we get the other linker, $D$-linker. Like the encoder, the linkers also consist of $6$ layers of building blocks, with a bump that extends to the seventh layer. The building blocks in the $1\times 3$ red rectangle in Figure \ref{fig_3d_linker} are used to indicate how the layers are aligned. The building block with a red rectangle is on top of the building blocks in the red rectangle in the previous layer. So a linker is made up of a main body of $1\times 3\times 6$ cuboid of building blocks and some additional building blocks attached to the main body.


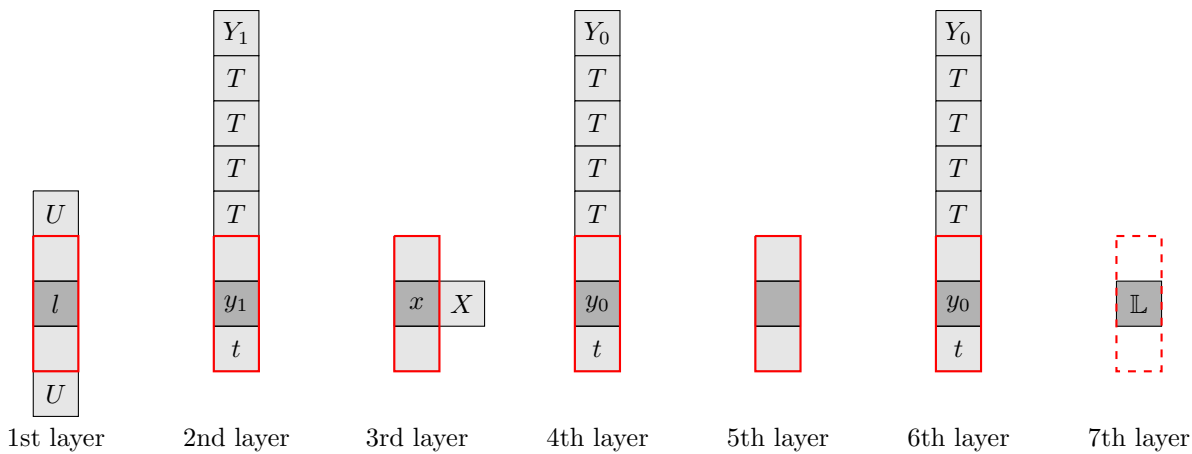
\begin{figure}[H]
\begin{center}
\begin{tikzpicture}[scale=0.6]

\foreach \x in {0,4, 8, 12, 16, 20}
\foreach \y in {0} 
{
\draw [ fill=gray!20] (\x+0,\y+0)--(\x+1,\y+0)--(\x+1,\y+3)--(\x+0,\y+3)--(\x+0,\y+0);

\draw [ fill=gray!60] (\x+0,\y+1)--(\x+1,\y+1)--(\x+1,\y+2)--(\x+0,\y+2)--(\x+0,\y+1);

\draw (\x+0,\y+1)--(\x+1,\y+1);
\draw (\x+0,\y+2)--(\x+1,\y+2);
}

\foreach \x in {0}
\foreach \y in {0} 
{
\draw [ fill=gray!20] (\x+0,\y+0)--(\x+1,\y+0)--(\x+1,\y-1)--(\x+0,\y-1)--(\x+0,\y+0);
\draw [ fill=gray!20] (\x+0,\y+4)--(\x+1,\y+4)--(\x+1,\y+3)--(\x+0,\y+3)--(\x+0,\y+4);
}
\foreach \x in {4,12,20}
\foreach \y in {0,1,2,3,4} 
{
\draw [ fill=gray!20] (\x+0,\y+4)--(\x+1,\y+4)--(\x+1,\y+3)--(\x+0,\y+3)--(\x+0,\y+4);
}

\foreach \x in {9}
\foreach \y in {1} 
{

\draw [ fill=gray!20] (\x+0,\y+1)--(\x+1,\y+1)--(\x+1,\y+0)--(\x+0,\y+0)--(\x+0,\y+1);
}

\foreach \x in {24}
\foreach \y in {1} 
{

\draw [ fill=gray!60] (\x+0,\y+1)--(\x+1,\y+1)--(\x+1,\y+0)--(\x+0,\y+0)--(\x+0,\y+1);
}

\node at (0.5,-1.5) {$1$st layer};  
\node at (4.5,-1.5) {$2$nd layer};  
\node at (8.5,-1.5) {$3$rd layer};  
\node at (12.5,-1.5) {$4$th layer};    
\node at (16.5,-1.5) {$5$th layer};  
\node at (20.5,-1.5) {$6$th layer};
\node at (24.5,-1.5) {$7$th layer};

\foreach \x in {0}
\foreach \y in {-1,3} 
{ 
\node at (\x+0.5,\y+0.5) {$U$};  
}
\foreach \x in {4,12,20}
\foreach \y in {0} 
{ 
\node at (\x+0.5,\y+0.5) {$t$};  
}

\foreach \x in {4,12,20}
\foreach \y in {3,4,5,6} 
{ 
\node at (\x+0.5,\y+0.5) {$T$};  
}

\foreach \x in {12,20}
\foreach \y in {1} 
{ 
\node at (\x+0.5,\y+0.5) {$y_0$};  
}

\foreach \x in {12,20}
\foreach \y in {7} 
{ 
\node at (\x+0.5,\y+0.5) {$Y_0$};  
}

\foreach \x in {4}
\foreach \y in {1} 
{ 
\node at (\x+0.5,\y+0.5) {$y_1$};  
}

\foreach \x in {4}
\foreach \y in {7} 
{ 
\node at (\x+0.5,\y+0.5) {$Y_1$};  
}

\foreach \x in {8}
\foreach \y in {1} 
{ 
\node at (\x+0.5,\y+0.5) {$x$};  
}

\foreach \x in {9}
\foreach \y in {1} 
{ 
\node at (\x+0.5,\y+0.5) {$X$};  
}
\foreach \x in {24}
\foreach \y in {1} 
{ 
\node at (\x+0.5,\y+0.5) {$\mathbb{L}$};  
}
\foreach \x in {0}
\foreach \y in {1} 
{ 
\node at (\x+0.5,\y+0.5) {$l$};  
}

\foreach \x in {0,4,8,12,16,20}
\foreach \y in {1} 
{
\draw [ thick, color=red] (\x+0,\y+2)--(\x+1,\y+2)--(\x+1,\y-1)--(\x+0,\y-1)--(\x+0,\y+2);
}

\foreach \x in {24}
\foreach \y in {1} 
{
\draw [ thick, dashed, color=red] (\x+0,\y+2)--(\x+1,\y+2)--(\x+1,\y-1)--(\x+0,\y-1)--(\x+0,\y+2);
}

\end{tikzpicture}
\end{center}
\caption{Level-2 layer diagram the $U$-linker.}\label{fig_3d_linker}
\end{figure}

In the first layer, a building block $U$ is attached to both the north and the south of the main body of the $U$-linker, and a building block $D$ is attached to both the north and the south of the main body of the $D$-linker. This is the only difference between the two linkers, all other building blocks are exactly the same for both the $U$-linker and $D$-linker.

In the second layer, an array of four building blocks $T$ and a building block $Y_1$ is attached to the north side of the main body, and there are a building block $t$ and a building block $y_1$ in the main body. The fourth layer and the sixth layer are almost identical to the second layer, except that the building blocks $Y_1$ and $y_1$ are replaced by $Y_0$ and $y_0$, respectively. These building blocks $Y_0$, $Y_1$, $y_0$ and $y_1$ only appear in the linkers. Therefore, they ensure that to the north (or to the south) of a linker in any tiling of the space, there must be another linker whose main body is exactly $3$ building blocks away from the main body of the former linker. Furthermore, the two linkers are vertically aligned. In other words, the first layers of the two linkers are at the same altitude, the second layers of the two linkers are at the same altitude, and so on. In the third layer, a building block $X$ is attached to the east side of the main body, and there is a building block $x$ in the main body. As the building blocks $X$ and $x$ only appear in the linkers, they ensure that on the east or west of a linker, there must be another linker whose main body is adjacent to the main body of the former one, and they are vertically aligned too (i.e, the first layers of both linkers are on the same altitude, the second layers of both linkers are on the same altitude, and so on). In the fifth layer, there are three functional cubes. In the seventh layer, a building block $\mathbb{L}$ is attached to the top side of the main body (thus this building block $\mathbb{L}$ forms the seventh layer by itself). The first layer has a building block $l$ that matches $\mathbb{L}$ in the main body. For a similar reason, on top (or below) a linker, there must be another linker, and the main body of the two linkers are adjacent.

In general, for a set of $p$ Wang tiles, the corresponding linkers consist of a main body of $1\times 3\times 2p$ building blocks. Two building blocks $U$ and two building blocks $D$ are attached to the first layer of the main body of a $U$-linker and a $D$-linker, respectively. Therefore, the first layer is called the \textit{matching} layer, as it always matches a building block $u$ with another building block $u$, or a building block $d$ with another $d$.  Like the example in Figure \ref{fig_3d_linker}, building blocks $Y_1$ and $y_1$ are used in the second layer, and building blocks $Y_0$ and $y_0$ are used in the $2i$-th layer for $2\leq i\leq p$. As one of the purposes of the building blocks $T$ is to ensure the entanglement between the linkers and the encoders, these $2i$-th ($1\leq i\leq p$) layers are called \textit{interlocking} layers. All other $p-1$ layers, i.e. $(2j+1)$-th ($j=1,\cdots, p-1$) layers, are called \textit{padding} layers. The building blocks $X$ and $x$ only appear in the third layer (we can assume that $p\geq 2$). The building blocks $\mathbb{L}$ and $l$ appear in the $(2p+1)$-th layer and the first layer, respectively. Like the example in Figure~\ref{fig_3d_linker}, the building blocks $X$, $x$, $Y_0$, $y_0$, $Y_1$, $y_1$, $\mathbb{L}$, and $l$ ensure that the linkers must form a rigid structure in any tiling of the space. We describe this rigid structure more precisely below. We call the point at the southwest corner of the bottom face of the unique building block $l$ of a linker (either a $U$-linker or a $D$-linker) the \textit{representative point} of the linker. Without loss of generality, we may assume that in any tiling, the representative point of one of the linkers is the origin of the space, then the rigid structure of the linkers in any tiling can be stated in the following fact.

\begin{Fact}\label{fct_3d_linker}
    In any tiling, the set of representative points of all the linkers forms a lattice $$\{(10x,60y,20pz)|x,y,z \in \mathbb{Z}\}.$$
\end{Fact}

Note that there is also a kind of flexibility in the rigid structure stated by Fact \ref{fct_3d_linker}. Each linker in the lattice structure can be chosen to be a $U$-linker or a $D$-linker, independent of each other.

As we mentioned in the previous subsection, a freely movable building block $U$ (or $D$) is a \textit{filler} tile. In all, we have four tiles: one encoder, two linkers ($U$-linker and $D$-linker) and one filler. It is easy to check that all four tiles are connected polycubes.

\subsection{Tiling Pattern and Proof of Theorem \ref{thm_3d_new}}

\begin{proof}[Proof of Theorem \ref{thm_3d_new}]
We prove by reduction from Wang's domino problem which is known to be undecidable by Theorem \ref{thm_berger}. For each instance of Wang's domino problem, namely a set of Wang tiles (assuming the instance has at least two Wang tiles as Wang's domino problem remains undecidable in this case), we have already constructed a set of four tiles in the previous subsection. So it remains to show that the set of Wang tiles can tile the plane if and only if the corresponding set of four polycubes can tile the space.

\begin{itemize}
    \item (\textbf{Linkers must be used.}) First of all, in order to tile the entire space, the linkers must be used. In fact, if the encoder is used, then the linker must be used, because the building blocks $t$ in the encoder can only be matched by the building block $T$ in the linkers. If the filler is used, then the encoder must be used, as the fillers cannot tile the space by themselves and can only be matched by the building blocks $u$ or $d$ that appear exclusively in the encoder. In all cases, the linkers have to be used in the end.

    \item (\textbf{Place encoders in the gaps left by the rigid yet flexible lattice structure of the linkers.}) Since the linkers must be used, they form a lattice structure by Fact \ref{fct_3d_linker} of the previous subsection. A significant part of space left by the lattice formed by the linkers is the gaps between the main bodies of the linkers and their neighboring linkers to the south or to the north (note that there are no gaps between the main body of a linker to the main body of its neighboring linkers to the east or to the west, nor its neighboring linkers above or below). Due to the existence of building blocks $T$ between the main bodies of the linkers in the lattice structure, we can only place the encoders in the gaps.
    

    \item (\textbf{Flexibility in choosing an encoding layer of the encoders to be aligned with a matching layer.}) When filling the gaps left by the lattice of linkers with the encoders, because of the specific purpose building blocks $t$ and $T$, the non-encoding layers of the encoders must be aligned with the interlocking layers of the linkers, and the encoding layers of the encoders must be aligned with the matching layers or padding layers of the linkers. Because the encoder and linker have the same number of layers ($6$ layers in the example and $2p$ layers in general), exactly one of the encoding layers (i.e. a simulated Wang tile) is aligned with a matching layer. By Fact \ref{fct_column}, encoders form vertical columns. Therefore, for each column, it is the same simulated Wang tile (the same encoding layer within each encoder) that is aligned with a matching layer in the lattice of linkers. By moving up and down, each column of encoders has the flexibility of choosing which encoding layer to be aligned with a matching layer. Different columns can choose which encoding layer to align with a matching layer independently.


\begin{figure}[H]
\begin{center}
\begin{tikzpicture}[scale=0.3,pattern1/.style={draw=orange!70,pattern color=orange!70, pattern=north east lines}]

\foreach \x in {0,16}
\foreach \y in {0,12} 
{
\draw [fill=gray!20] (\x,\y)--(\x+16,\y)--(\x+16,\y+3)--(\x,\y+3)--(\x,\y);
}

\foreach \x in {-8,8,24}
\foreach \y in {6} 
{
\draw [fill=gray!20] (\x,\y)--(\x+16,\y)--(\x+16,\y+3)--(\x,\y+3)--(\x,\y);
}

\foreach \x in {0,...,31}
\foreach \y in {3,9} 
{
\draw [fill=orange!20] (\x,\y)--(\x+1,\y)--(\x+1,\y+3)--(\x,\y+3)--(\x,\y);
}

\foreach \x in {0,...,31}
\foreach \y in {2,6,8,12} 
{
\filldraw [ pattern1] (\x+0,\y+0)--(\x+1,\y+0)--(\x+1,\y+1)--(\x+0,\y+1)--(\x+0,\y+0);
}

\foreach \x in {0,16}
\foreach \y in {0,12} 
{
\draw  (\x,\y)--(\x+16,\y)--(\x+16,\y+3)--(\x,\y+3)--(\x,\y);
}

\foreach \x in {-8,8,24}
\foreach \y in {6} 
{
\draw  (\x,\y)--(\x+16,\y)--(\x+16,\y+3)--(\x,\y+3)--(\x,\y);
}

\foreach \x in {0,16}
\foreach \y in {2,6} 
{
\node at (\x+0.5,\y+0.5) {$u$};  
\node at (\x+1.5,\y+0.5) {$u$};  
\node at (\x+2.5,\y+0.5) {$u$};  
\node at (\x+3.5,\y+0.5) {$u$};  
}

\foreach \x in {8,24}
\foreach \y in {8,12} 
{
\node at (\x+0.5,\y+0.5) {$u$};  
\node at (\x+1.5,\y+0.5) {$u$};  
\node at (\x+2.5,\y+0.5) {$u$};  
\node at (\x+3.5,\y+0.5) {$u$};  
}

\foreach \x in {4,20}
\foreach \y in {8,12} 
{
\node at (\x+0.5,\y+0.5) {$d$};  
\node at (\x+1.5,\y+0.5) {$d$};  
\node at (\x+2.5,\y+0.5) {$d$};  
\node at (\x+3.5,\y+0.5) {$d$};  
}

\foreach \x in {12,28}
\foreach \y in {2,6} 
{
\node at (\x+0.5,\y+0.5) {$d$};  
\node at (\x+1.5,\y+0.5) {$d$};  
\node at (\x+2.5,\y+0.5) {$d$};  
\node at (\x+3.5,\y+0.5) {$d$};  
}

\draw [thick] (20+17.5,-15+16)--(20+18.5,-14+16)--(20+19.5,-15+16)--(20+18.5,-16+16)--(20+17.5,-15+16);
\draw [thick] (20+21.5,-15+16)--(20+20.5,-14+16)--(20+19.5,-15+16)--(20+20.5,-16+16)--(20+21.5,-15+16);
\draw [thick] (20+16.5,-14+16)--(20+17.5,-13+16)--(20+18.5,-14+16)--(20+17.5,-15+16)--(20+16.5,-14+16);
\draw [thick] (20+20.5,-14+16)--(20+19.5,-13+16)--(20+18.5,-14+16)--(20+19.5,-15+16)--(20+20.5,-14+16);
\draw [thick] (20+20.5,-14+16)--(20+21.5,-13+16)--(20+22.5,-14+16)--(20+21.5,-15+16)--(20+20.5,-14+16);
\draw [thick] (20+17.5,-13+16)--(20+18.5,-14+16)--(20+19.5,-13+16)--(20+18.5,-12+16)--(20+17.5,-13+16);
\draw [thick] (20+21.5,-13+16)--(20+20.5,-14+16)--(20+19.5,-13+16)--(20+20.5,-12+16)--(20+21.5,-13+16);

\draw [color=black!20] (20+16.5,-14+16)--(20+22.5,-14+16);
\draw [color=black!20] (20+17.5,-15+16)--(20+21.5,-15+16);
\draw [color=black!20] (20+17.5,-13+16)--(20+21.5,-13+16);

\draw [color=black!20] (20+18.5,-16+16)--(20+18.5,-12+16);
\draw [color=black!20] (20+20.5,-16+16)--(20+20.5,-12+16);

\draw [color=black!20] (20+21.5,-15+16)--(20+21.5,-13+16);
\draw [color=black!20] (20+19.5,-15+16)--(20+19.5,-13+16);
\draw [color=black!20] (20+17.5,-15+16)--(20+17.5,-13+16);

\end{tikzpicture}
\end{center}
\caption{The tiling pattern in the matching layers.}\label{fig_3d_pattern}
\end{figure}
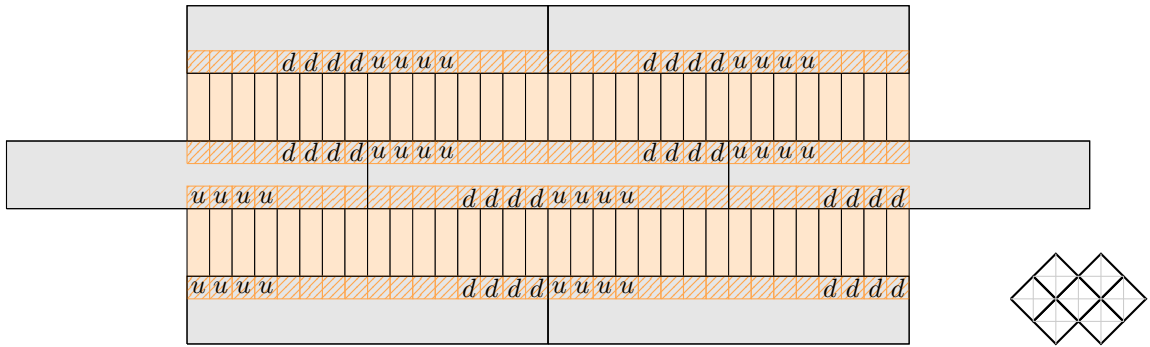

\item (\textbf{The tiling pattern in the matching layers.}) Now let us focus on the matching layers in the lattice of linkers. We claim that in order to tile the entire space, the matching layer must form the pattern as illustrated in Figure \ref{fig_3d_pattern}. By the encoding method we describe in the previous subsection, all the color edges of Wang tiles are simulated by four consecutive building blocks beginning with $d$ and ending with $u$ in the encoding layers. So in the matching layer, the unique four consecutive building blocks $d$ (four consecutive building blocks $u$, resp.) on the south side of an encoding layer can only be matched to the unique four consecutive building blocks $d$ (four consecutive building blocks $u$, resp.) on the north side of an encoding layer of another encoder. This immediately results in the pattern illustrated in Figure \ref{fig_3d_pattern}.

\item (\textbf{Tilability in the non-matching layers.}) Next, let us check all the layers other than the matching layers for the tilability of the space. The non-encoding layers of the encoders and the interlocking layers of the linkers are perfectly matched, and they fill up those layers without overlap. The small gaps between the padding layers of the linkers and the encoding layers of the encoders can be exactly filled by the fillers.

\item (\textbf{Tilability in the matching layers.}) Finally, the tilability of the set of four polycubes depends on the matching layers. By Fact \ref{fct_column}, all the matching layers have to be identical, so we just need to investigate the tilability of one matching layer. A matching layer fills up the space exactly if and only if all the building blocks $u$ and $d$ in the matching layer of the encoders can be matched by the linkers. By the encoding method, this is equivalent to the fact that the simulated Wang tiles can tile the plane.
\end{itemize}
The above arguments show that the set of four polycubes, one encoder, one linker, and two fillers, simulates the tiling of a set of Wang tiles perfectly. In other words, the set of four polycubes constructed from a set of Wang tiles can tile the space if and only if the corresponding set of Wang tiles can tile the plane. This completes the proof.
\end{proof}

We conclude this section by summarizing the novel techniques in the construction of the set of four polycubes, compared to the construction in \cite{yz24c}. Each of these techniques alone may not be able to reduce the number of tiles to four yet retaining undecidability, but they combine to achieve this goal. First, the use of a multiconnected tile, the encoder. The number of tunnels in the encoder increases infinitely as the number of Wang tiles and the number of colors increases. Second, more redundancy is added to the construction. A non-encoding layer is added between every pair of adjacent encoding layers. The encoding method for the colors of the Wang tiles uses two extra bits by adding a prefix $0$ and a suffix $1$. Third, with the redundancy of the encoding method, the rigidity of the tiling pattern in the matching layer is forced in a unified method with the color matching of simulated Wang tiles. Using the same mechanism for two different purposes definitely has the potential of saving one or two tiles. The linkers also play two (if not three or more) roles in the tiling. One is to simulate the matching rules of Wang tiles, the other is to form a rigid lattice structure allowing the encoders to choose a simulated Wang tiles with respect to this rigid structure.

\section{Undecidability of Tiling $4$-dimensional Space}\label{sec_4d}

In this section, by applying the lifting technique introduced in \cite{yz24b,yz24c}, we further reduced the number of tiles by one. See Subsection $3.2$ of \cite{yz24c} for a high-level description of the lifting technique. However, there is a subtle difference in the current paper when applying the lifting technique. In our previous works \cite{yz24b,yz24c}, all the bumps and dents of the building blocks are lifted to a higher dimension. In this paper, only the bumps and dents of some of the building blocks are lifted to the fourth dimension (mostly those building blocks that are used to encode the colors). The bumps and dents of other building blocks remain essentially within $3$-dimensional space. Using different forms of bumps and dents in different dimensions is a more efficient way to exploit the dimensions of the $4$-dimensional space.

\subsection{Building Blocks}

We interpret the fourth dimension as time and use the term \textit{frame} to denote the discrete unit of time. Define $10$ consecutive frames to be a \textit{slice}. Therefore, the $4$-dimensional space is sometimes referred to as \textit{spacetime}, and the subspace of the first three dimensions is referred to as \textit{space} in this section. A $10\times 10\times 10\times 10$ polyhypercube is called a \textit{functional hypercube} in the $4$-dimensional spacetime. So, a $4$-dimensional functional hypercube is a $3$-dimensional functional cube (a $10\times 10\times 10$ polycube) that remains unchanged for $10$ frames. In this section, a building block is a subset of a functional hypercube.

If a $3$-dimensional building block introduced in the previous section remains unchanged for $10$ frames, we get a $4$-dimensional building block, which is called a \textbf{thick} (thick in the fourth dimension, namely time) version of its $3$-dimensional counterpart. By abuse of notation, we use the same letter to denote both a $3$-dimensional building block and its thick version in $4$-dimensional spacetime. For example, in this section, $t$ means a thick version of the building block illustrated in Figure \ref{fig_3d_t}. So the dents and bumps of these thick building blocks remain in the $3$-dimensional space.

In addition to the thick version counterpart of the $3$-dimensional building blocks, we also introduced some real $4$-dimensional building blocks which change over time. In other words, these are the building blocks whose dents and bumps are lifted to a higher dimension. Let $K$ be a $3$-dimensional $10\times 10\times 10$ polycube, so $K$ is a set of $1000$ unit cubes. Let $T_1$ denote the set of unit cubes on the outer surface of $K$. For $i=2,3,4,5$, let $T_i$ be the set of unit cubes on the outer surface of $K-\bigcup_{k=1}^{i-1} T_k$. Thus, we have partitioned $K$ into disjoint union of $5$ sets: $T_i$, $(1\leq i\leq 5)$. We define three building blocks $c^{(4)}$, $C^{(4)}$ and $D^{(4)}$ for encoding the color of Wang tiles as follows, by listing their ten $3$-dimensional $10\times 10\times 10$ frames in order:
\begin{itemize}
\item $c^{(4)}: \emptyset, T_1\cup T_2\cup T_3 \cup T_5, T_1 \cup T_3 \cup T_5,    T_1 \cup T_5, K, K, K, K, K, K;$
\item
$C^{(4)}: K, T_4, T_2 \cup T_4,   T_2 \cup T_3 \cup T_4, \emptyset, \emptyset, \emptyset,  \emptyset, \emptyset, \emptyset;$
\item 
$D^{(4)}: \emptyset, \emptyset, \emptyset, \emptyset, \emptyset, K,  T_4, T_2 \cup T_4,   T_2 \cup T_3 \cup T_4, \emptyset.$
\end{itemize}

Note that the building blocks $C^{(4)}$ and $D^{(4)}$ are identical except that $C^{(4)}$ lives in the first $5$ frames of a building block and $D^{(4)}$ lives in the last $5$ frames. This will make a difference when $C^{(4)}$ and $D^{(4)}$ are attached to other building blocks in a bigger tile. The building block $c^{(4)}$ is also the \textit{filler}, which is one of the three tiles that we will construct in the next subsection. Note that we use a superscript $^{(4)}$ to indicate that the dents and bumps of these building blocks are in the fourth dimension.

The next two pairs of building blocks $v^{(4)}$ and $V^{(4)}$, and $w^{(4)}$ and $W^{(4)}$ will be used to ensure the alignment of the encoders with respect to time:
\begin{itemize}
    \item 
$v^{(4)}: \emptyset, T_1\cup T_2\cup T_3 \cup T_4, T_1 \cup T_3 \cup T_4, T_1, K, K, K, K, K, K;$
 \item 
$V^{(4)}: K, T_5, T_2 \cup T_5,   T_2 \cup T_3 \cup T_4  \cup  T_5, \emptyset, \emptyset, \emptyset, \emptyset, \emptyset, \emptyset;$
 \item 
$w^{(4)}: \emptyset, T_1\cup T_2\cup T_3 \cup T_4, T_1\cup T_2\cup T_4, T_1 \cup T_2, K, K, K, K, K, K;$
 \item 
$W^{(4)}: K, T_5, T_3 \cup T_5, T_3 \cup T_4  \cup  T_5, \emptyset, \emptyset, \emptyset, \emptyset, \emptyset, \emptyset.$
\end{itemize}

Finally, the building block $\mathbb{E}^{(4)}$ is defined as a building block with $14$ frames for convenience. In other words, it is contained in a $10\times 10\times 10\times 14$ polyhypercube:
\begin{itemize}
    \item $\mathbb{E}^{(4)}:   T_1\cup T_2\cup T_3, T_1\cup T_2\cup T_3, T_1 \cup T_3 , T_1 , K, K, K, K, K, K, T_4 \cup T_5, T_4 \cup T_5, T_2 \cup T_4 \cup T_5,  T_2 \cup  T_3 \cup T_4  \cup  T_5. $
\end{itemize}

The building block $\mathbb{E}^{(4)}$ is used to ensure that the tiling of the $4$-dimensional space is essentially the same for every slice (recall that $1$ slice consists of $10$ consecutive frames).

It is easy to check that all the above $8$ newly defined $4$-dimensional building blocks are connected in spacetime.

\subsection{The Set of Three Tiles}

In this subsection, we describe the method of constructing a set of three tiles in $4$-dimensional space for every set of Wang tiles, by taking the set of Wang tiles in Figure \ref{fig_w3} as an example. They are basically all lifted from the set of $3$-dimensional tiles defined in the previous section.

The first tile is the \textit{encoder} illustrated in Figure \ref{fig_4d_encoder}. It is a thick version of the $3$-dimensional encoder we defined in the previous section, with a few modifications. So the encoder is lifted from the $3$-dimensional encoder. As the encoder is a tile in the $4$-dimensional space, Figure \ref{fig_4d_encoder} in fact depicts its projection to the subspace of the first three dimensions. As all the building blocks are made of $10$ frames (except $\mathbb{E}^{(4)}$ which has $14$ frames), they lie in the same slice in the fourth dimension (except the last $4$ frames of $\mathbb{E}^{(4)}$ extend to the next slice). In other words, the $i$-th frame of one building block is aligned with the $i$-th frame of another building block, for $1\leq i\leq 10$. Therefore, Figure \ref{fig_4d_encoder} gives a complete definition of the encoder.

More than half of the building blocks remain unchanged over time. The $3$-dimensional building blocks $u$ and $d$ are replaced by (or lifted to) $D^{(4)}$ and $C^{(4)}$, respectively. This is the most important change. The bumps and dents of the building blocks that are related to encoding the colors are lifted to the fourth dimension. The encoding method is essentially unchanged, and $0$ and $1$ is now physically represented by $C^{(4)}$ and $D^{(4)}$, respectively. The pair of building blocks $\mathbb{E}$ and $e$ are replaced by (or lifted to) $W^{(4)}$ and $w^{(4)}$, respectively. A pair of building blocks $V^{(4)}$ and $v^{(4)}$ are added to each encoding layer.  A building block $\mathbb{E}^{(4)}$ is added to the first layer. The building blocks $W^{(4)}$, $w^{(4)}$, $V^{(4)}$, $v^{(4)}$ and $\mathbb{E}^{(4)}$ are added or lifted for the issues related to the alignment with time (i.e. the fourth dimension).


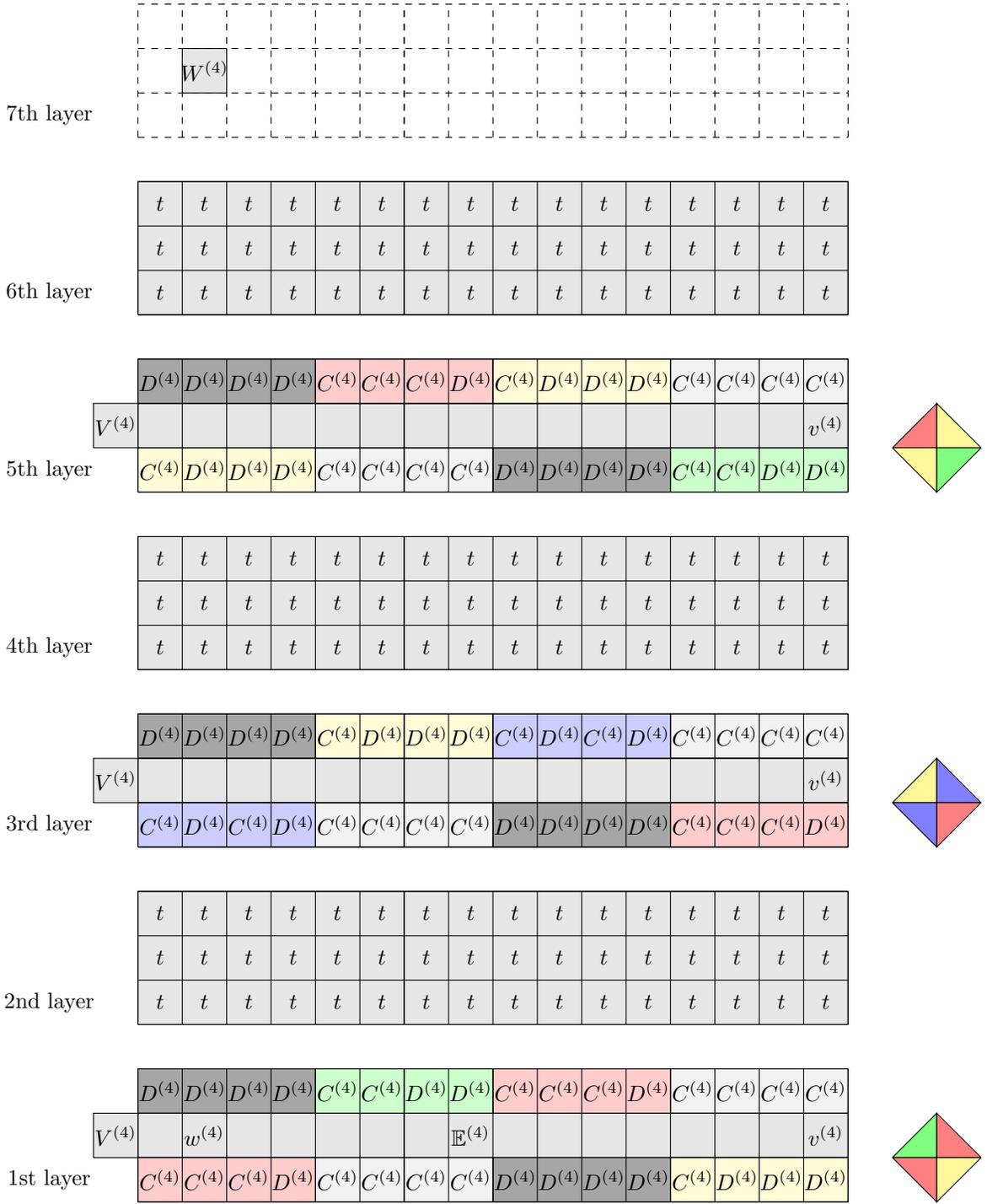
\begin{figure}[H]
\begin{center}
\begin{tikzpicture}[scale=0.7,pattern1/.style={draw=red,pattern color=gray!70, pattern=north east lines}]

\foreach \x in {0}
\foreach \y in {2,10,18} 
{
\draw [ fill=gray!70] (\x+0,\y+0)--(\x+4,\y+0)--(\x+4,\y+1)--(\x+0,\y+1)--(\x+0,\y+0);
\node at (\x+0.5,\y+0.5) {$D^{(4)}$};  
\node at (\x+1.5,\y+0.5) {$D^{(4)}$};  
\node at (\x+2.5,\y+0.5) {$D^{(4)}$};  
\node at (\x+3.5,\y+0.5) {$D^{(4)}$};  
}
\foreach \x in {8}
\foreach \y in {0,8,16} 
{
\draw [ fill=gray!70] (\x+0,\y+0)--(\x+4,\y+0)--(\x+4,\y+1)--(\x+0,\y+1)--(\x+0,\y+0);
\node at (\x+0.5,\y+0.5) {$D^{(4)}$};  
\node at (\x+1.5,\y+0.5) {$D^{(4)}$};  
\node at (\x+2.5,\y+0.5) {$D^{(4)}$};  
\node at (\x+3.5,\y+0.5) {$D^{(4)}$};  
}
\foreach \x in {4}
\foreach \y in {0,8,16} 
{
\draw [ fill=gray!10 ] (\x+0,\y+0)--(\x+4,\y+0)--(\x+4,\y+1)--(\x+0,\y+1)--(\x+0,\y+0);
\node at (\x+0.5,\y+0.5) {$C^{(4)}$};  
\node at (\x+1.5,\y+0.5) {$C^{(4)}$};  
\node at (\x+2.5,\y+0.5) {$C^{(4)}$};  
\node at (\x+3.5,\y+0.5) {$C^{(4)}$};  
}
\foreach \x in {12}
\foreach \y in {2,10,18} 
{
\draw [ fill=gray!10 ] (\x+0,\y+0)--(\x+4,\y+0)--(\x+4,\y+1)--(\x+0,\y+1)--(\x+0,\y+0);
\node at (\x+0.5,\y+0.5) {$C^{(4)}$};  
\node at (\x+1.5,\y+0.5) {$C^{(4)}$};  
\node at (\x+2.5,\y+0.5) {$C^{(4)}$};  
\node at (\x+3.5,\y+0.5) {$C^{(4)}$};  
}

\foreach \x in {0}
\foreach \y in {1,9,17} 
{
\draw [ fill=gray!20] (\x+0,\y+0)--(\x+16,\y+0)--(\x+16,\y+1)--(\x+0,\y+1)--(\x+0,\y+0);
}
\foreach \x in {0}
\foreach \y in {4,12,20} 
{
\draw [ fill=gray!20] (\x+0,\y+0)--(\x+16,\y+0)--(\x+16,\y+3)--(\x+0,\y+3)--(\x+0,\y+0);
}


\foreach \x in {0}
\foreach \y in {0} 
{
\draw [ fill=red!20] (\x+0,\y+0)--(\x+4,\y+0)--(\x+4,\y+1)--(\x+0,\y+1)--(\x+0,\y+0);
\node at (\x+0.5,\y+0.5) {$C^{(4)}$};  
\node at (\x+1.5,\y+0.5) {$C^{(4)}$};  
\node at (\x+2.5,\y+0.5) {$C^{(4)}$};  
\node at (\x+3.5,\y+0.5) {$D^{(4)}$};  
}
\foreach \x in {8}
\foreach \y in {2} 
{
\draw [ fill=red!20] (\x+0,\y+0)--(\x+4,\y+0)--(\x+4,\y+1)--(\x+0,\y+1)--(\x+0,\y+0);
\node at (\x+0.5,\y+0.5) {$C^{(4)}$};  
\node at (\x+1.5,\y+0.5) {$C^{(4)}$};  
\node at (\x+2.5,\y+0.5) {$C^{(4)}$};  
\node at (\x+3.5,\y+0.5) {$D^{(4)}$};  
}

\foreach \x in {12}
\foreach \y in {8} 
{
\draw [ fill=red!20] (\x+0,\y+0)--(\x+4,\y+0)--(\x+4,\y+1)--(\x+0,\y+1)--(\x+0,\y+0);
\node at (\x+0.5,\y+0.5) {$C^{(4)}$};  
\node at (\x+1.5,\y+0.5) {$C^{(4)}$};  
\node at (\x+2.5,\y+0.5) {$C^{(4)}$};  
\node at (\x+3.5,\y+0.5) {$D^{(4)}$};  
}

\foreach \x in {4}
\foreach \y in {18} 
{
\draw [ fill=red!20] (\x+0,\y+0)--(\x+4,\y+0)--(\x+4,\y+1)--(\x+0,\y+1)--(\x+0,\y+0);
\node at (\x+0.5,\y+0.5) {$C^{(4)}$};  
\node at (\x+1.5,\y+0.5) {$C^{(4)}$};  
\node at (\x+2.5,\y+0.5) {$C^{(4)}$};  
\node at (\x+3.5,\y+0.5) {$D^{(4)}$};  
}


\foreach \x in {12}
\foreach \y in {0} 
{
\draw [ fill=yellow!20] (\x+0,\y+0)--(\x+4,\y+0)--(\x+4,\y+1)--(\x+0,\y+1)--(\x+0,\y+0);
\node at (\x+0.5,\y+0.5) {$C^{(4)}$};  
\node at (\x+1.5,\y+0.5) {$D^{(4)}$};  
\node at (\x+2.5,\y+0.5) {$D^{(4)}$};  
\node at (\x+3.5,\y+0.5) {$D^{(4)}$};  
}
\foreach \x in {4}
\foreach \y in {10} 
{
\draw [ fill=yellow!20] (\x+0,\y+0)--(\x+4,\y+0)--(\x+4,\y+1)--(\x+0,\y+1)--(\x+0,\y+0);
\node at (\x+0.5,\y+0.5) {$C^{(4)}$};  
\node at (\x+1.5,\y+0.5) {$D^{(4)}$};  
\node at (\x+2.5,\y+0.5) {$D^{(4)}$};  
\node at (\x+3.5,\y+0.5) {$D^{(4)}$};  
}

\foreach \x in {0}
\foreach \y in {16} 
{
\draw [ fill=yellow!20] (\x+0,\y+0)--(\x+4,\y+0)--(\x+4,\y+1)--(\x+0,\y+1)--(\x+0,\y+0);
\node at (\x+0.5,\y+0.5) {$C^{(4)}$};  
\node at (\x+1.5,\y+0.5) {$D^{(4)}$};  
\node at (\x+2.5,\y+0.5) {$D^{(4)}$};  
\node at (\x+3.5,\y+0.5) {$D^{(4)}$};  
}

\foreach \x in {8}
\foreach \y in {18} 
{
\draw [ fill=yellow!20] (\x+0,\y+0)--(\x+4,\y+0)--(\x+4,\y+1)--(\x+0,\y+1)--(\x+0,\y+0);
\node at (\x+0.5,\y+0.5) {$C^{(4)}$};  
\node at (\x+1.5,\y+0.5) {$D^{(4)}$};  
\node at (\x+2.5,\y+0.5) {$D^{(4)}$};  
\node at (\x+3.5,\y+0.5) {$D^{(4)}$};  
}

\foreach \x in {4}
\foreach \y in {2} 
{
\draw [ fill=green!20] (\x+0,\y+0)--(\x+4,\y+0)--(\x+4,\y+1)--(\x+0,\y+1)--(\x+0,\y+0);
\node at (\x+0.5,\y+0.5) {$C^{(4)}$};  
\node at (\x+1.5,\y+0.5) {$C^{(4)}$};  
\node at (\x+2.5,\y+0.5) {$D^{(4)}$};  
\node at (\x+3.5,\y+0.5) {$D^{(4)}$};  
}

\foreach \x in {12}
\foreach \y in {16} 
{
\draw [ fill=green!20] (\x+0,\y+0)--(\x+4,\y+0)--(\x+4,\y+1)--(\x+0,\y+1)--(\x+0,\y+0);
\node at (\x+0.5,\y+0.5) {$C^{(4)}$};  
\node at (\x+1.5,\y+0.5) {$C^{(4)}$};  
\node at (\x+2.5,\y+0.5) {$D^{(4)}$};  
\node at (\x+3.5,\y+0.5) {$D^{(4)}$};  
}

\foreach \x in {8}
\foreach \y in {10} 
{
\draw [ fill=blue!20] (\x+0,\y+0)--(\x+4,\y+0)--(\x+4,\y+1)--(\x+0,\y+1)--(\x+0,\y+0);
\node at (\x+0.5,\y+0.5) {$C^{(4)}$};  
\node at (\x+1.5,\y+0.5) {$D^{(4)}$};  
\node at (\x+2.5,\y+0.5) {$C^{(4)}$};  
\node at (\x+3.5,\y+0.5) {$D^{(4)}$};  
}
\foreach \x in {0}
\foreach \y in {8} 
{
\draw [ fill=blue!20] (\x+0,\y+0)--(\x+4,\y+0)--(\x+4,\y+1)--(\x+0,\y+1)--(\x+0,\y+0);
\node at (\x+0.5,\y+0.5) {$C^{(4)}$};  
\node at (\x+1.5,\y+0.5) {$D^{(4)}$};  
\node at (\x+2.5,\y+0.5) {$C^{(4)}$};  
\node at (\x+3.5,\y+0.5) {$D^{(4)}$};  
}


\foreach \x in {0}
\foreach \y in {0,...,23} 
{ 
\draw  (\x+0,0+\y)--(\x+16,0+\y);
}

\foreach \x in {0}
\foreach \y in {0,4,8,12,16,20} 
{ 
\draw  (\x,0+\y)--(\x,3+\y);
\draw  (\x+10,0+\y)--(\x+10,3+\y);
\draw  (\x+9,0+\y)--(\x+9,3+\y);
\draw  (\x+8,0+\y)--(\x+8,3+\y);
\draw  (\x+7,0+\y)--(\x+7,3+\y);
\draw  (\x+6,0+\y)--(\x+6,3+\y);
\draw  (\x+5,0+\y)--(\x+5,3+\y);
\draw  (\x+4,0+\y)--(\x+4,3+\y);
\draw  (\x+3,0+\y)--(\x+3,3+\y);
\draw  (\x+2,0+\y)--(\x+2,3+\y);
\draw  (\x+1,0+\y)--(\x+1,3+\y);
\draw  (\x+6,0+\y)--(\x+6,3+\y);
\draw  (\x+15,0+\y)--(\x+15,3+\y);
\draw  (\x+14,0+\y)--(\x+14,3+\y);
\draw  (\x+13,0+\y)--(\x+13,3+\y);
\draw  (\x+12,0+\y)--(\x+12,3+\y);
\draw  (\x+11,0+\y)--(\x+11,3+\y);
\draw  (\x+16,0+\y)--(\x+16,3+\y);
}

\foreach \x in {0}
\foreach \y in {24} 
{ 
\draw [dashed]  (0,\y)--(16,\y);\draw [dashed]  (0,\y+3)--(16,\y+3);
\draw [dashed]  (0,\y+2)--(16,\y+2); \draw [dashed]  (0,\y+1)--(16,\y+1);

\draw [dashed]  (\x,0+\y)--(\x,3+\y);
\draw [dashed]  (\x+10,0+\y)--(\x+10,3+\y);
\draw [dashed]  (\x+9,0+\y)--(\x+9,3+\y);
\draw [dashed]  (\x+8,0+\y)--(\x+8,3+\y);
\draw [dashed]  (\x+7,0+\y)--(\x+7,3+\y);
\draw [dashed]  (\x+6,0+\y)--(\x+6,3+\y);
\draw [dashed]  (\x+5,0+\y)--(\x+5,3+\y);
\draw [dashed]  (\x+4,0+\y)--(\x+4,3+\y);
\draw [dashed]  (\x+3,0+\y)--(\x+3,3+\y);
\draw [dashed]  (\x+2,0+\y)--(\x+2,3+\y);
\draw [dashed]  (\x+1,0+\y)--(\x+1,3+\y);
\draw [dashed]  (\x+6,0+\y)--(\x+6,3+\y);
\draw [dashed]  (\x+15,0+\y)--(\x+15,3+\y);
\draw [dashed]  (\x+14,0+\y)--(\x+14,3+\y);
\draw [dashed]  (\x+13,0+\y)--(\x+13,3+\y);
\draw [dashed]  (\x+12,0+\y)--(\x+12,3+\y);
\draw [dashed]  (\x+11,0+\y)--(\x+11,3+\y);
\draw [dashed]  (\x+16,0+\y)--(\x+16,3+\y);
}

\node at (-2,0.5) {$1$st layer};  
\node at (-2,4.5) {$2$nd layer};  
\node at (-2,8.5) {$3$rd layer};  
\node at (-2,12.5) {$4$th layer};    
\node at (-2,16.5) {$5$th layer};  
\node at (-2,20.5) {$6$th layer};
\node at (-2,24.5) {$7$th layer};

\foreach \x in {17}
\foreach \y in {1} 
{ 
\draw [ fill=green!50] (\x+0,0+\y)--(\x+1,1+\y)--(\x+1,0+\y)--(\x+0,0+\y);
\draw [ fill=red!50] (\x+0,0+\y)--(\x+1,-1+\y)--(\x+1,0+\y)--(\x+0,0+\y);
\draw [ fill=red!50] (\x+2,0+\y)--(\x+1,1+\y)--(\x+1,0+\y)--(\x+2,0+\y);
\draw [ fill=yellow!50] (\x+2,0+\y)--(\x+1,-1+\y)--(\x+1,0+\y)--(\x+2,0+\y);
}

\foreach \x in {17}
\foreach \y in {9} 
{ 
\draw [ fill=yellow!50] (\x+0,0+\y)--(\x+1,1+\y)--(\x+1,0+\y)--(\x+0,0+\y);
\draw [ fill=blue!50] (\x+0,0+\y)--(\x+1,-1+\y)--(\x+1,0+\y)--(\x+0,0+\y);
\draw [ fill=blue!50] (\x+2,0+\y)--(\x+1,1+\y)--(\x+1,0+\y)--(\x+2,0+\y);
\draw [ fill=red!50] (\x+2,0+\y)--(\x+1,-1+\y)--(\x+1,0+\y)--(\x+2,0+\y);
}

\foreach \x in {17}
\foreach \y in {17} 
{ 
\draw [ fill=red!50] (\x+0,0+\y)--(\x+1,1+\y)--(\x+1,0+\y)--(\x+0,0+\y);
\draw [ fill=yellow!50] (\x+0,0+\y)--(\x+1,-1+\y)--(\x+1,0+\y)--(\x+0,0+\y);
\draw [ fill=yellow!50] (\x+2,0+\y)--(\x+1,1+\y)--(\x+1,0+\y)--(\x+2,0+\y);
\draw [ fill=green!50] (\x+2,0+\y)--(\x+1,-1+\y)--(\x+1,0+\y)--(\x+2,0+\y);
}

\foreach \x in {0,...,15}
\foreach \y in {4,5,6,12,13,14,20,21,22} 
{ 
\node at (\x+0.5,\y+0.5) {$t$};  
}

\foreach \x in {1}
\foreach \y in {25} 
{
\draw [ fill=gray!20] (\x+0,\y+0)--(\x+1,\y+0)--(\x+1,\y+1)--(\x+0,\y+1)--(\x+0,\y+0);
\node at (\x+0.5,\y+0.5) {$W^{(4)}$};  
}

\foreach \x in {1}
\foreach \y in {1} 
{
\node at (\x+0.5,\y+0.5) {$w^{(4)}$};  
}
\foreach \x in {7}
\foreach \y in {1} 
{
\node at (\x+0.5,\y+0.5) {$\mathbb{E}^{(4)}$};  
}

\foreach \x in {-1}
\foreach \y in {1,9,17} 
{
\draw [ fill=gray!20] (\x+0,\y+0)--(\x+1,\y+0)--(\x+1,\y+1)--(\x+0,\y+1)--(\x+0,\y+0);
\node at (\x+0.5,\y+0.5) {$V^{(4)}$};  \node at (\x+16.5,\y+0.5) {$v^{(4)}$};  
}

\end{tikzpicture}
\end{center}
\caption{$3$-dimensional projection of the encoder.}\label{fig_4d_encoder}
\end{figure}

Similarly, the \textit{linker} is a thick version lifted from its $3$-dimensional counterpart defined in the previous section, with just one modification. Like the $4$-dimensional encoder, since all the building blocks lie in the same slice, it suffices to depict their projection to the $3$-dimensional subspace of the first three dimensions in Figure \ref{fig_4d_linker}. Unlike the previous section in which we have two linkers, thanks to the lifting technique, we only need a single linker in the $4$-dimensional case, as both the building blocks $U$ and $D$ are replaced with (or lifted to) the building block $c^{(4)}$.

As we have mentioned in the previous subsection, the third tile, the \textit{filler}, is just the building block $c^{(4)}$. So we have completed the construction of the set of three tiles: an encoder, a linker, and a filler. It can be checked that all three tiles are connected.


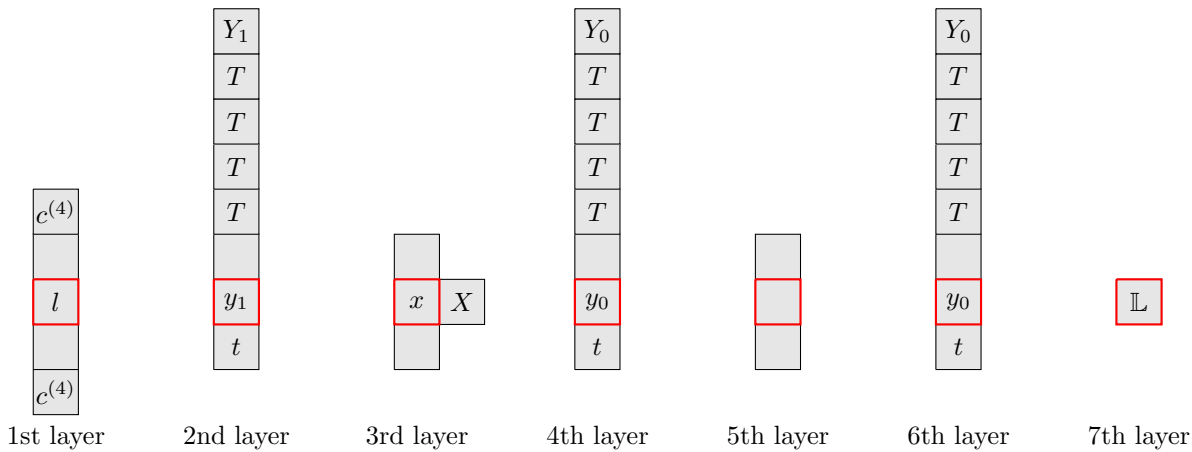
\begin{figure}[H]
\begin{center}
\begin{tikzpicture}[scale=0.6]

\foreach \x in {0,4, 8, 12, 16, 20}
\foreach \y in {0} 
{
\draw [ fill=gray!20] (\x+0,\y+0)--(\x+1,\y+0)--(\x+1,\y+3)--(\x+0,\y+3)--(\x+0,\y+0);

\draw (\x+0,\y+1)--(\x+1,\y+1);
\draw (\x+0,\y+2)--(\x+1,\y+2);
}

\foreach \x in {0}
\foreach \y in {0} 
{
\draw [ fill=gray!20] (\x+0,\y+0)--(\x+1,\y+0)--(\x+1,\y-1)--(\x+0,\y-1)--(\x+0,\y+0);
\draw [ fill=gray!20] (\x+0,\y+4)--(\x+1,\y+4)--(\x+1,\y+3)--(\x+0,\y+3)--(\x+0,\y+4);
}
\foreach \x in {4,12,20}
\foreach \y in {0,1,2,3,4} 
{
\draw [ fill=gray!20] (\x+0,\y+4)--(\x+1,\y+4)--(\x+1,\y+3)--(\x+0,\y+3)--(\x+0,\y+4);
}

\foreach \x in {9}
\foreach \y in {1} 
{

\draw [ fill=gray!20] (\x+0,\y+1)--(\x+1,\y+1)--(\x+1,\y+0)--(\x+0,\y+0)--(\x+0,\y+1);
}

\foreach \x in {24}
\foreach \y in {1} 
{

\draw [ fill=gray!20] (\x+0,\y+1)--(\x+1,\y+1)--(\x+1,\y+0)--(\x+0,\y+0)--(\x+0,\y+1);
}

\node at (0.5,-1.5) {$1$st layer};  
\node at (4.5,-1.5) {$2$nd layer};  
\node at (8.5,-1.5) {$3$rd layer};  
\node at (12.5,-1.5) {$4$th layer};    
\node at (16.5,-1.5) {$5$th layer};  
\node at (20.5,-1.5) {$6$th layer};
\node at (24.5,-1.5) {$7$th layer};

\foreach \x in {0}
\foreach \y in {-1,3} 
{ 
\node at (\x+0.5,\y+0.5) {$c^{(4)}$};  
}
\foreach \x in {4,12,20}
\foreach \y in {0} 
{ 
\node at (\x+0.5,\y+0.5) {$t$};  
}

\foreach \x in {4,12,20}
\foreach \y in {3,4,5,6} 
{ 
\node at (\x+0.5,\y+0.5) {$T$};  
}

\foreach \x in {12,20}
\foreach \y in {1} 
{ 
\node at (\x+0.5,\y+0.5) {$y_0$};  
}

\foreach \x in {12,20}
\foreach \y in {7} 
{ 
\node at (\x+0.5,\y+0.5) {$Y_0$};  
}

\foreach \x in {4}
\foreach \y in {1} 
{ 
\node at (\x+0.5,\y+0.5) {$y_1$};  
}

\foreach \x in {4}
\foreach \y in {7} 
{ 
\node at (\x+0.5,\y+0.5) {$Y_1$};  
}

\foreach \x in {8}
\foreach \y in {1} 
{ 
\node at (\x+0.5,\y+0.5) {$x$};  
}

\foreach \x in {9}
\foreach \y in {1} 
{ 
\node at (\x+0.5,\y+0.5) {$X$};  
}
\foreach \x in {24}
\foreach \y in {1} 
{ 
\node at (\x+0.5,\y+0.5) {$\mathbb{L}$};  
}
\foreach \x in {0}
\foreach \y in {1} 
{ 
\node at (\x+0.5,\y+0.5) {$l$};  
}

\foreach \x in {0,4,8,12,16,20,24}
\foreach \y in {1} 
{
\draw [ thick, color=red] (\x+0,\y+1)--(\x+1,\y+1)--(\x+1,\y+0)--(\x+0,\y+0)--(\x+0,\y+1);
}

\end{tikzpicture}
\end{center}
\caption{$3$-dimensional projection of the linker.}\label{fig_4d_linker}
\end{figure}

\subsection{Proof of Theorem \ref{thm_main}}
\begin{proof}[Proof of Theorem \ref{thm_main}]
For any set of Wang tiles, we have already constructed a set of three tiles in the $4$-dimensional space in the previous subsection. To complete the proof, it suffices to show that the set of three tiles can tile the $4$-dimensional space if and only if the corresponding set of Wang tiles can tile the plane.

\begin{itemize}
    \item (\textbf{Linkers must be used.}) By the same argument as the proof of Theorem \ref{thm_3d_new}, the linkers must be used in order to tile the $4$-dimensional space.

    \item (\textbf{Linkers form rigid lattice structure in space, but flexible in time.}) By slightly more involved argument than that of Fact \ref{fct_3d_linker}, the linkers must form a rigid lattice structure in the $3$-dimensional space. In fact, by considering a slice of spacetime, the linkers must form a rigid lattice structure within this slice. As we will see in the next few paragraphs, the gaps between the main bodies of the lattice structure must be filled by the encoders. The building blocks $\mathbb{E}^{(4)}$ ensure that the encoders must be the same in every slice. This in turn ensures that the $3$-dimensional lattice structure of the linkers must be rigid as time flows (i.e., the lattice structure remains the same in every slice).

    However, each linker in the lattice has the flexibility to slide along the direction of time (the fourth dimension). Linkers at distinct locations of space can slide along the direction of time independently of each other. Linkers at the same location of the space can slide consistently. Note the difference in the flexibility of the linkers between the $3$-dimensional case and the $4$-dimensional case. In the $3$-dimensional case in the previous section, the linkers are flexible as they can be chosen from the two linkers: $D$-linker and $U$-linker.

    \item (\textbf{Matching encoders with linkers.}) For the same reason, the $3$-dimensional space left by the linkers (in some slices) can only be filled by encoders. Note that because the linkers can slide in time, locally in the spacetime, it can not only match a building block $C^{(4)}$ with another $C^{(4)}$ but also match a building block $D^{(4)}$ with another $D^{(4)}$. This is the good side, one linker for two jobs, which saves us a tile. But there is also a bad side, the linkers now may even match a building block $C^{(4)}$ with a building block $D^{(4)}$ locally. In this situation, the two building blocks $C^{(4)}$ and $D^{(4)}$ are not aligned in time, and they are not in the same slice. We will soon show that this is impossible as all the encoders have to be aligned in time.

    \item (\textbf{The encoders must be aligned in time.}) Suppose the contrast that the four consecutive building blocks $D^{(4)}$ on the south side of an encoding layer of an encoder in the matching layer is matched to the four consecutive building blocks $C^{(4)}$ on the north side of an encoding layer of another encoder (see Figure \ref{fig_4d_align}). Then we get an immediate contradiction between the four consecutive building blocks $C^{(4)}$ on the south side of the first encoder and the corresponding building blocks on the north side of the second encoder (see the building blocks marked $*$ in Figure \ref{fig_4d_align}).


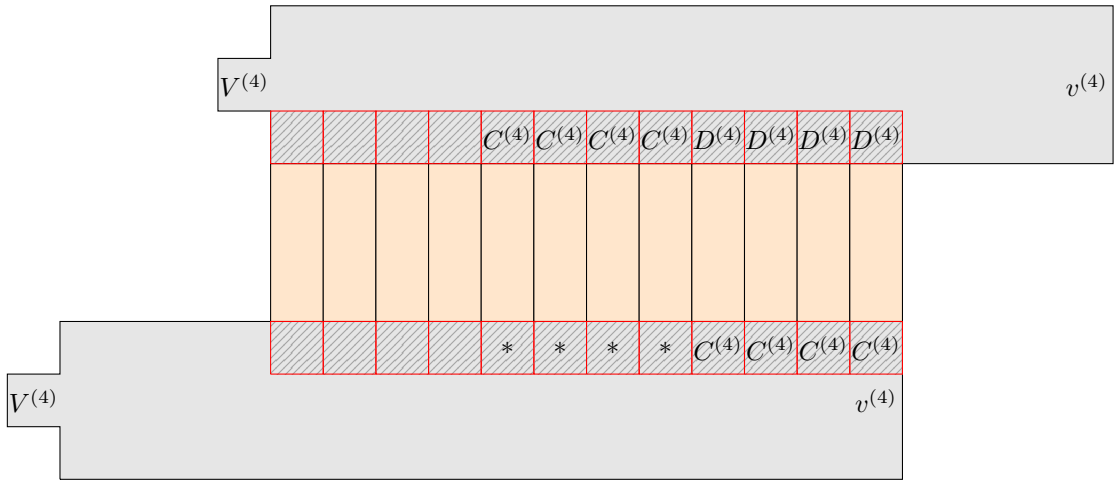
\begin{figure}[H]
\begin{center}
\begin{tikzpicture}[scale=0.7,pattern1/.style={draw=red,pattern color=gray!70, pattern=north east lines}]

\foreach \x in {-4}
\foreach \y in {0} 
{
\draw [fill=gray!20] (\x,\y)--(\x+16,\y)--(\x+16,\y+3)--(\x,\y+3)--(\x,\y+2)--(\x-1,\y+2)--(\x-1,\y+1)--(\x,\y+1)--(\x,\y);
}

\foreach \x in {0}
\foreach \y in {6} 
{
\draw [fill=gray!20] (\x,\y)--(\x+16,\y)--(\x+16,\y+3)--(\x,\y+3)--(\x,\y+2)--(\x-1,\y+2)--(\x-1,\y+1)--(\x,\y+1)--(\x,\y);
}

\foreach \x in {0,...,11}
\foreach \y in {3} 
{
\draw [fill=orange!20] (\x,\y)--(\x+1,\y)--(\x+1,\y+3)--(\x,\y+3)--(\x,\y);
}

\foreach \x in {0,...,11}
\foreach \y in {2,6} 
{
\filldraw [ pattern1] (\x+0,\y+0)--(\x+1,\y+0)--(\x+1,\y+1)--(\x+0,\y+1)--(\x+0,\y+0);
}

\foreach \x in {8}
\foreach \y in {6} 
{
\node at (\x+0.5,\y+0.5) {$D^{(4)}$};  
\node at (\x+1.5,\y+0.5) {$D^{(4)}$};  
\node at (\x+2.5,\y+0.5) {$D^{(4)}$};  
\node at (\x+3.5,\y+0.5) {$D^{(4)}$};  
}

\foreach \x in {4}
\foreach \y in {6} 
{
\node at (\x+0.5,\y+0.5) {$C^{(4)}$};  
\node at (\x+1.5,\y+0.5) {$C^{(4)}$};  
\node at (\x+2.5,\y+0.5) {$C^{(4)}$};  
\node at (\x+3.5,\y+0.5) {$C^{(4)}$};  
}

\foreach \x in {8}
\foreach \y in {2} 
{
\node at (\x+0.5,\y+0.5) {$C^{(4)}$};  
\node at (\x+1.5,\y+0.5) {$C^{(4)}$};  
\node at (\x+2.5,\y+0.5) {$C^{(4)}$};  
\node at (\x+3.5,\y+0.5) {$C^{(4)}$};  
}
\foreach \x in {4}
\foreach \y in {2} 
{
\node at (\x+0.5,\y+0.5) {$*$};  
\node at (\x+1.5,\y+0.5) {$*$};  
\node at (\x+2.5,\y+0.5) {$*$};  
\node at (\x+3.5,\y+0.5) {$*$};  
}

\foreach \x in {-5}
\foreach \y in {1} 
{
\node at (\x+0.5,\y+0.5) {$V^{(4)}$};  
}
\foreach \x in {-1}
\foreach \y in {7} 
{
\node at (\x+0.5,\y+0.5) {$V^{(4)}$};  
}

\foreach \x in {11}
\foreach \y in {1} 
{
\node at (\x+0.5,\y+0.5) {$v^{(4)}$};  
}
\foreach \x in {15}
\foreach \y in {7} 
{
\node at (\x+0.5,\y+0.5) {$v^{(4)}$};  
}

\end{tikzpicture}
\end{center}
\caption{$3$-dimensional projection of the linking of two encoders in the matching layer.}\label{fig_4d_align}
\end{figure}


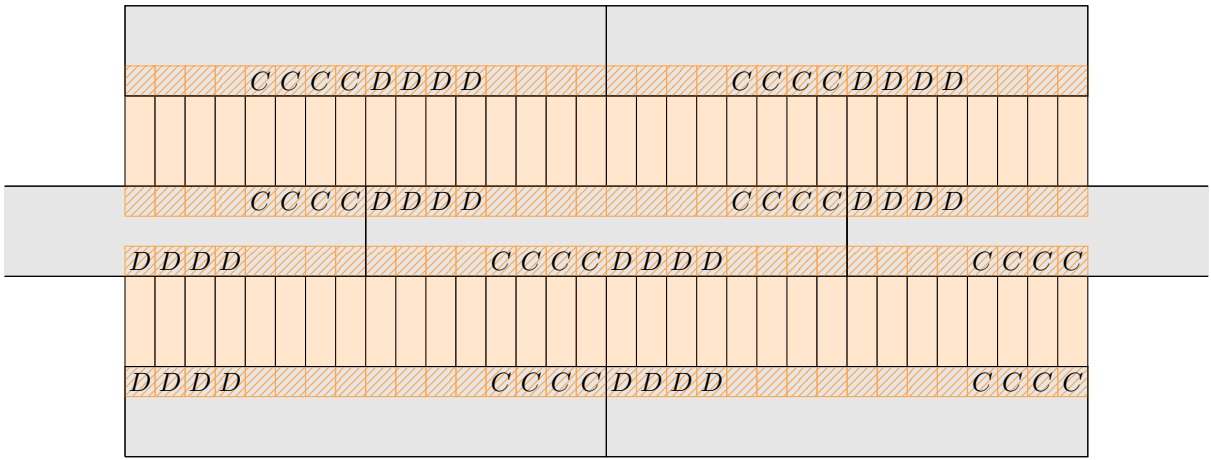
\begin{figure}[H]
\begin{center}
\begin{tikzpicture}[scale=0.4,pattern1/.style={draw=orange!70,pattern color=orange!70, pattern=north east lines}]

\foreach \x in {0,16}
\foreach \y in {0,12} 
{
\draw [fill=gray!20] (\x,\y)--(\x+16,\y)--(\x+16,\y+3)--(\x,\y+3)--(\x,\y);
}

\foreach \x in {-8}
\foreach \y in {6} 
{
\draw [fill=gray!20] (\x+4,\y)--(\x+16,\y)--(\x+16,\y+3)--(\x+4,\y+3);
}
\foreach \x in {8}
\foreach \y in {6} 
{
\draw [fill=gray!20] (\x,\y)--(\x+16,\y)--(\x+16,\y+3)--(\x,\y+3)--(\x,\y);
}
\foreach \x in {24}
\foreach \y in {6} 
{
\draw [fill=gray!20] (\x+12,\y)--(\x,\y)--(\x,\y+3)--(\x+12,\y+3);
}

\foreach \x in {0,...,31}
\foreach \y in {3,9} 
{
\draw [fill=orange!20] (\x,\y)--(\x+1,\y)--(\x+1,\y+3)--(\x,\y+3)--(\x,\y);
}

\foreach \x in {0,...,31}
\foreach \y in {2,6,8,12} 
{
\filldraw [ pattern1] (\x+0,\y+0)--(\x+1,\y+0)--(\x+1,\y+1)--(\x+0,\y+1)--(\x+0,\y+0);
}

\foreach \x in {0,16}
\foreach \y in {0,12} 
{
\draw  (\x,\y)--(\x+16,\y)--(\x+16,\y+3)--(\x,\y+3)--(\x,\y);
}

\foreach \x in {8}
\foreach \y in {6} 
{
\draw  (\x,\y)--(\x+16,\y)--(\x+16,\y+3)--(\x,\y+3)--(\x,\y);
}
\foreach \x in {-8}
\foreach \y in {6} 
{
\draw  (\x+4,\y)--(\x+16,\y)--(\x+16,\y+3)--(\x+4,\y+3);
}
\foreach \x in {24}
\foreach \y in {6} 
{
\draw  (\x+12,\y)--(\x,\y)--(\x,\y+3)--(\x+12,\y+3);
}

\foreach \x in {0,16}
\foreach \y in {2,6} 
{
\node at (\x+0.5,\y+0.5) {$D$};  
\node at (\x+1.5,\y+0.5) {$D$};  
\node at (\x+2.5,\y+0.5) {$D$};  
\node at (\x+3.5,\y+0.5) {$D$};  
}

\foreach \x in {8,24}
\foreach \y in {8,12} 
{
\node at (\x+0.5,\y+0.5) {$D$};  
\node at (\x+1.5,\y+0.5) {$D$};  
\node at (\x+2.5,\y+0.5) {$D$};  
\node at (\x+3.5,\y+0.5) {$D$};  
}

\foreach \x in {4,20}
\foreach \y in {8,12} 
{
\node at (\x+0.5,\y+0.5) {$C$};  
\node at (\x+1.5,\y+0.5) {$C$};  
\node at (\x+2.5,\y+0.5) {$C$};  
\node at (\x+3.5,\y+0.5) {$C$};  
}

\foreach \x in {12,28}
\foreach \y in {2,6} 
{
\node at (\x+0.5,\y+0.5) {$C$};  
\node at (\x+1.5,\y+0.5) {$C$};  
\node at (\x+2.5,\y+0.5) {$C$};  
\node at (\x+3.5,\y+0.5) {$C$};  
}

\end{tikzpicture}
\end{center}
\caption{$3$-dimensional projection of the tiling pattern in the matching layers.}\label{fig_4d_pattern}
\end{figure}

So after the building blocks for encoding the colors have been lifted to $4$-dimensional, the four consecutive building blocks $D^{(4)}$ of one encoder must be matched to the four consecutive $D^{(4)}$ of another encoder in the matching layers. For the same reason, four consecutive building blocks $C^{(4)}$ of one encoder must be matched to the four consecutive $C^{(4)}$ of another encoder in the matching layers. Therefore, we still have the same tiling pattern in the matching layers (see Figure \ref{fig_4d_pattern}, we omit the superscript $^{(4)}$ for building blocks $C^{(4)}$ and $D^{(4)}$ in this figure). As a consequence, two neighboring encoders in the south-north direction must be aligned in time. The building blocks $V^{(4)}$ and $v^{(4)}$ ensure that two neighboring encoders in the east-west direction are aligned in time. The building blocks $W^{(4)}$ and $w^{(4)}$ ensure that two neighboring encoders in the vertical direction are aligned in time. In all, the encoders are aligned with time. In exact words, given two arbitrary encoders in a tiling of the spacetime, they are either in two disjoint time slices or in the same slice of $10$ frames (ignoring the extra $4$ frames of $\mathbb{E}^{(4)}$).

\item (\textbf{The distributions of encoders are identical in every slice.}) For any encoder in a tiling, the building block $\mathbb{E}^{(4)}$ ensures that there must be another encoder in the same position of space in the next time slice and the previous time slice. So the $4$-dimensional spacetime is tiled by the translations of one slice to the past and the future, while each slice is more or less a thick version of the $3$-dimensional tiling we obtained in the previous section. This is a key feature of the lifting technique.

\item (\textbf{The patterns formed by the encoding layers are the same for every matching layer.}) Within a time slice, the encoders must form two-way infinite vertical columns by the pair of building blocks $W^{(4)}$ and $w^{(4)}$. Therefore, the tilings of the encoders are periodic in the vertical direction.

\item (\textbf{Tilability of the set of three tiles.}) Finally, we check whether the set of three tiles can tile the $4$-dimensional spacetime. For the non-encoding layers of the encoders and the interlocking layers of the linkers, the $4$-dimensional tiles are exactly the same as the $3$-dimensional tiles, they fill up the space together and remain unchanged over time. Similarly to the $3$-dimensional case, the gaps between the padding layers of the linkers and the encoding layers of the encoders can be exactly filled by the fillers, and this can be done in the same way for every slice. Therefore, the tilability of the spacetime depends on the tilability in the matching layers. As we have argued, the matching layer must form the pattern as illustrated in Figure \ref{fig_4d_pattern}. Our linker can now be translated in the direction of time to match two building blocks of the same type. For two building blocks $C^{(4)}$, the linkers can matched them by being aligned in time with regard to the encoders (i.e. in the same slice as the encoders). For two building blocks $D^{(4)}$, the linkers can match them by being translated $5$ frames to the future or to the past, with regard to the time frames of the encoders. Therefore, the set of three tiles can tile the spacetime if and only if building blocks of the encoding layers of the encoders exposed to the matching layers are all matched by the linkers. This in turn is equivalent to that the matching layers simulate a tiling of the corresponding set of the Wang tiles. 
\end{itemize}
Thus the set of three tiles can tile the $4$-dimensional space if and only if the corresponding set of the Wang tiles can tile the $2$-dimensional plane. This completes the proof.
\end{proof}

\section{Conclusion}\label{sec_conclu}
In this paper, we prove that translational tiling of $4$-dimensional space with a set of three connected tiles is undecidable by combining the techniques we developed in our previous works \cite{yang23,yang23b,yz24,yz24a,yz24b,yz24c} and several novel techniques we introduce in this paper. This is another step forward in the search for an answer to the following problem.

\begin{Problem}[Problem \ref{pro_main} with fixed parameters $(n,1)$]\label{pro_focus}
    Is there a fixed positive integer $n$ such that translational tiling of $\mathbb{Z}^n$ with a monotile is undecidable?
\end{Problem}

With our result in this paper that translational tiling is undecidable for three tiles, we are just two steps away from a full solution to Problem \ref{pro_focus}. So it is very tempting to continue working on the undecidability of translational tiling with one or two tiles of a fixed space $\mathbb{Z}^n$.

We conclude by making a comparison of the undecidability between the translational tiling problem and the general tiling problem (i.e. tiles are allowed to be rotated or reflected in the general tiling problem). For the $2$-dimensional case, the translation tiling problem is known to be decidable with a single tile. But it is still open whether it is decidable to tile the plane with a single tile allowing rotation. The recent discovery of aperiodic monotiles \cite{smith23a,smith23b} suggests that the general tiling problem with a monotile may be undecidable in the plane. In this sense, the general tiling problem is more complex than the translational tiling problem. Demaine and Langerman showed that the general tiling of the plane with three polygons is undecidable \cite{dl24}. This is the best known undecidability result for the general tiling problem with respect to the number of tiles. It is also very tempting to investigate the undecidability of the general tiling problem of higher dimensional spaces with one or two tiles.

\section*{Acknowledgements}
The first author was supported by the Research Fund of Guangdong University of Foreign Studies (Nos. 297-ZW200011 and 297-ZW230018), and the National Natural Science Foundation of China (No. 61976104).



\end{document}